\documentclass[11pt]{amsart}

\usepackage{epsfig}
\usepackage{xy}
\usepackage{array}

\usepackage{stmaryrd}

\usepackage{graphicx,color}
\usepackage{xcolor}
\usepackage{tikz}
\usetikzlibrary{arrows,calc}
\usepackage{etex}

\usepackage{mathdots}

\usepackage{amsmath,amssymb}
\usepackage{float}
\usepackage{graphics}

\usepackage{pdflscape} 

\xyoption{all}

\newtheorem{lemma}{Lemma}[section]
\newtheorem{theorem}[lemma]{Theorem}
\newtheorem{corollary}[lemma]{Corollary}

\newtheorem{proposition}[lemma]{Proposition}

\newtheorem{remark}[lemma]{Remark}
\newtheorem{definition}[lemma]{Definition}

\newcommand{\id}{{\operatorname{id}}}
\newcommand{\Hom}{\operatorname{Hom}}
\newcommand{\Ext}{\operatorname{Ext}}
\newcommand{\Sub}{\operatorname{Sub}}
\newcommand{\End}{\operatorname{End}}

\newcommand{\add}{\operatorname{add}}
\newcommand{\irr}{\operatorname{Irr}}

\newcommand{\module}{\operatorname{mod}}
\newcommand{\C}{\mathcal C}

\newcommand{\ac}{\mathcal{A}}
\newcommand{\cat}{\mathcal{C}}
\newcommand{\dc}{\mathcal{D}}
\newcommand{\uc}{\mathcal{U}}
\newcommand{\qc}{\mathcal{Q}}
\newcommand{\rc}{\mathcal{R}}

\newcommand{\qcc}{\qc_{\mathcal{R}_{\cat}}}
\newcommand{\qcd}{\qc_{\mathcal{R}_{\dc}}}

\newcommand{\db}[2]{D^\text{b}\!\left(\field \mathrm{#1}_{#2}\right)}
\newcommand{\shift}{\Sigma}
\newcommand{\susm}{\shift^{-1}}

\newcommand{\fl}{\rightarrow}
\newcommand{\gfl}{\longrightarrow}

\newcommand{\al}{\alpha}

\newcommand{\ld}{\lambda}
\newcommand{\vph}{\varphi}

\newcommand{\nb}{\mathbb{N}}
\newcommand{\zb}{\mathbb{Z}}

\newcommand{\dr}{\ar@{->}[r]}
\newcommand{\bdr}{\ar@{->}[dr]}
\newcommand{\hdr}{\ar@{->}[ur]}

\newcommand{\tbar}{\overline{T}}

\newcommand{\Endt}{\End_{\cat}(T)^\text{op}}
\newcommand{\modt}{\module\Endt}
\newcommand{\field}{\mathbb{K}}

\begin{document}

\title[Algebras
  from maximal rigid objects]{Algebras of finite representation type arising from maximal rigid objects}

\author[Buan]{Aslak Bakke Buan}
\address{
Department of Mathematical Sciences,
Norwegian University of Science and Technology,
7491 Trondheim,
NORWAY 
}
\email{aslakb@math.ntnu.no}

\author[Palu]{Yann Palu}
\address{
LAMFA, Facult\'e des sciences, 
33 rue Saint-Leu, 
80039 Amiens Cedex 1,
FRANCE
}
\email{yann.palu@u-picardie.fr}

\author[Reiten]{Idun Reiten}
\address{
Department of Mathematical Sciences,
Norwegian University of Science and Technology,
7491 Trondheim,
NORWAY 
}
\email{idunr@math.ntnu.no}


\begin{abstract}
We give a complete classification of all algebras appearing as 
endomorphism algebras of maximal rigid objects in standard $2$-Calabi-Yau 
categories of finite type. Such categories are equivalent to 
certain orbit categories of derived categories of Dynkin algebras.
It turns out that with one exception, all
the algebras that occur are $2$-Calabi-Yau-tilted, and therefore appear in an
earlier classification by Bertani-{\O}kland and Oppermann. We explain
this phenomenon by investigating the subcategories generated by rigid objects
in standard $2$-Calabi-Yau 
categories of finite type. 
\end{abstract}

\thanks{
This work was supported by two grants from 
the Norwegian Research Council (FRINAT grant numbers 196600 and  231000.) Support by the
Institut Mittag-Leffler (Djursholm, Sweden) is gratefully acknowledged.
Y.P. wishes to thank the Department of Mathematical Sciences, NTNU, Trondheim,
for their kind hospitality.
}

\maketitle

\section*{Introduction}

Motivated by trying to categorify the essential ingredients in the definition of cluster algebras by Fomin and Zelevinsky,
the authors of \cite{bmrrt} introduced the cluster category $\C_Q$ associated with a finite acyclic quiver $Q$.
The notion was later generalized by Amiot \cite{a}, dealing with quivers which are not necessarily acyclic.
Let $\field$ be an algebraically closed field of characteristic
zero.
Cluster categories are special cases of Hom-finite, triangulated 2-Calabi-Yau $\field$-categories (2-CY categories). In such categories, 
the cluster tilting objects, or more generally, maximal rigid objects,
play a special role for the categorification of cluster algebras. For cluster categories in the acyclic case these two classes coincide,
but in general maximal rigid objects in 2-CY categories are not necessarily cluster tilting.

The cluster-tilted algebras are the finite dimensional algebras
obtained as endomorphism algebras of cluster tilting objects
in cluster categories. These, and the more general class of
2-Calabi-Yau-tilted algebras, are of independent interest, and have
been studied by many authors, see \cite{k2, rei,rei2}.
As a natural generalization, one also considers the endomorphism
algebras of maximal rigid objects in 2-CY categories, here called {\em 2-endorigid algebras}.

When $\Gamma = \End_{\C}(T)$ is a 2-Calabi-Yau-tilted algebra, it is not known if the category $\C$ is determined by $\Gamma$, but this is known
to be true in the case of acyclic cluster categories \cite{kr2}. However, if we consider  2-endorigid algebras, then one frequently obtains the same algebras starting with different
 2-CY categories. In this paper we investigate this phenomenon. We restrict to the case where the 2-CY categories in question only have a finite number of isomorphism 
classes of indecomposable objects. Also in this case, it is known that the 2-endorigid algebras are of finite 
representation type \cite{iy}.
In \cite{a} and \cite{xz}, the structure of triangulated categories
with finitely many indecomposables was studied. Such categories have Serre functors, and
hence there is an associated AR-quiver.
Here  orbit categories of the form $D^b(\module \field Q)/ \vph$ play a special role, where $Q$ is a Dynkin quiver, and $D^b(\module \field Q)$ is the bounded derived category
of the path algebra $\field Q$. These are exactly
the standard categories, i.e. those which can be
identified with the mesh category of their AR-quiver, which are
2-CY and have only a finite number of indecomposable objects.
In \cite{bikr}, such orbit categories with the 2-CY property were classified. And 
as an application of that classification, the 2-CY-tilted algebras of finite representation type, coming from orbit
categories, were classified in \cite{bo} (one case was missed, as was
noticed in \cite{l}, see Section \ref{subsection: list} for details.)
These classifications are crucial for our investigations. Our main
result is a complete classification of the 2-endorigid algebras associated to standard
2-CY categories of finite type. In fact, we show that all such algebras, with one single exception, already appear in 
the classification of \cite{bo}. In order to prove this we show that
the following holds in almost all cases: 
If we fix a 2-CY (orbit) category $\C$ of finite type, then
there is an associated 2-CY category $\C'$ with cluster tilting objects, such that the full additive subcategories generated by the rigid objects 
in $\C$ and $\C'$ are equivalent.

It is known that in the case of standard 2-CY categories, the 2-CY-tilted algebras of finite representation type
are Jacobian \cite{bo} (when the algebraically closed field $\field$ is of characteristic zero):
There is a potential (i.e. a sum of cycles), such that the algebra is the Jacobian
of its Gabriel quiver with respect to this potential. Moreover, all Jacobians are 2-CY-tilted, by the work of Amiot \cite{a}.
However, as indicated, we point out that there is a 2-endorigid algebra which is not 2-CY-tilted, and therefore also not Jacobian.

In section 1, we give some background material on maximal rigid and cluster tilting objects. In section 2 we give our version of the classification of
2-CY orbit categories, and in particular we describe the rigid objects
in these categories.  Then, in section 3, we define functors identifying the subcategories of rigids
in the relevant cases. In section 4, we give the example of a 2-endorigid algebra of finite type which is not 2-CY-tilted.

\subsection*{Notation}
Unless stated otherwise, $\field$ will be an algebraically closed field of
characteristic zero.
We write $\shift$ for the shift functor in any orbit category, and $[1]$ for the shift in any derived category.
We will use the following notation:
\[\ac_{n,t} = \db{A}{(2t+1)(n+1)-3}/\tau^{t(n+1)-1}[1]\]
\[\dc_{n,t} = \db{D}{2t(n+1)}/\tau^{n+1}\vph^n,\]
where $\vph$ is induced by an automorphism of order 2 of ${\rm D}_{2t(n+1)}$.
The orbit categories that we consider are triangulated, by a theorem of Keller, see~\cite{k}.

\section{Background}

In this section, we give some background material on cluster tilting
and maximal rigid objects in $\Hom$-finite triangulated 2-Calabi-Yau
categories over an algebraically closed field $\field$.

Let $d$ be a non-negative integer. A $\Hom$-finite, triangulated $\field$-category $\C$ is called {\em
  d-Calabi-Yau} or d-CY for short, if we have a natural isomorphism 
$$D\Hom(X,Y) \simeq \Hom(Y,X[d])$$ 
for objects $X,Y$ in $\C$, where $D = \Hom_{\field}(\ ,\field)$ is the ordinary
$\field$-duality. 

A main example here is the cluster category $\C_Q$ associated with a
finite (connected) acyclic quiver $Q$ \cite{bmrrt}. Here $\C_Q$ is the
orbit category $D^b(\module \field Q)/ \tau^{-1}[1]$, where $\tau$
is the AR-translation on the bounded derived category $D^b(\module
\field Q)$. 
The cluster categories have been shown to be triangulated \cite{k}. Another
main example is the stable category $\underline{\module \Lambda}$,
where $\Lambda$ is a preprojective algebra of Dynkin type, investigated
in \cite{gls}.

An object $M$ in a triangulated category is called {\em rigid} if
$\Ext^1(M,M) = 0$, and {\em maximal rigid} if it is maximal with
respect to this property.
Let $\add M$ denote the additive closure of $M$.
If also $\Ext^1(M,X) =0$ implies $X\in\add M$,
then $M$ is said to be {\em cluster tilting}.
For the cluster categories $\C_Q$ and the stable module categories  
 $\underline{\module \Lambda}$, the maximal rigid objects are also 
cluster tilting \cite{bmrrt, gls}, but this is not the case in general. 

An object  $\overline{T}$ is called an {\em almost
complete cluster tilting} object in $\C_Q$, if there is an
indecomposable object $X$, not in $\add \tbar$, such that $\overline{T}
\amalg X$ is a cluster tilting object. 
It was shown in \cite{bmrrt} that if $\overline{T}$ is an almost
complete cluster tilting object in $\C_Q$, then there is a unique
indecomposable object $Y \not \simeq X$, such that $T^{\ast} =
\overline{T} \amalg Y$ is a cluster tilting object.

There is an interesting property for cluster tilting objects which does 
not hold for maximal rigid objects. For $T$ a cluster tilting object
in a $2$-CY category $\C$, there is an equivalence of categories 
$\C/\add T \to \module \End(T)$, by \cite{bmr,kr}.

For a connected 2-Calabi-Yau category, then either all maximal rigid
objects are cluster tilting, or none of them are \cite{zz}.
And if for a maximal rigid object $M$ there are no loops or 2-cycles
in the quiver of $\End(M)$, then $M$ is cluster tilting \cite{birs, xo}.

The main sources of examples for having maximal rigid objects which
are not cluster tilting are 1-dimensional hypersurface singularities
\cite{bikr} and cluster tubes \cite{bkl, bmv, v, y}.

The 2-Calabi-Yau-tilted algebras $\Gamma$ satisfy some nice
homological properties: They are Gorenstein of dimension $\leq 1$, and
$\Sub \Gamma$ is a Frobenius category whose stable category
$\underline{\Sub} \Gamma$ is
3-Calabi-Yau \cite{kr}. Here $\Sub \Gamma$ denotes the full additive
subcategory of $\module \Gamma$ generated by the submodules of objects
in $\add \Gamma$, and $\underline{\Sub} \Gamma$ denotes the corresponding stable
category, that is: the category with the same objects, but with $\Hom$-spaces given
as the $\Hom$-spaces in $\module \Gamma$ modulo maps factoring through projective objects.
By \cite{zz}, also the 2-endorigid algebras are Gorenstein of dimension $\leq 1$.

\section{Rigid objects in triangulated orbit categories of finite type}\label{section: rigid}

\subsection{The classification}\label{subsection: list}

In \cite{a}, Amiot classified all standard triangulated categories with finitely many
indecomposable objects. By using geometric descriptions in type $\rm{A}$ \cite{ccs}
and in type $\rm{D}$ \cite{s}, and direct computations in type $\rm{E}$, Burban--Iyama--Keller--Reiten
extracted from Amiot's list all 2-Calabi--Yau triangulated categories with cluster tilting objects,
and with non-zero maximal rigid objects (see the appendix of \cite{bikr}).
In this section, we give a restatement of the results in the appendix of \cite{bikr}.
We note two changes from their lists:
\begin{enumerate}
 \item[(L1)] The orbit category $\db{E}{8}/\tau^4$ has cluster tilting objects
(this case was first noticed by Ladkani in \cite{l});
 \item[(L2)] The orbit category $\db{D}{4}/\tau^2\vph$, where $\vph$ is induced by an automorphism of
 $D_4$ of order 2, has non-zero maximal rigid objects which are not cluster tilting.
\end{enumerate}

\begin{proposition}[Amiot ; Burban--Iyama--Keller--Reiten]\label{proposition: list CTO}
The standard, 2-Calabi--Yau, triangulated categories with finitely many indecomposable objects
and with cluster tilting objects are exactly the cluster categories of Dynkin types $A$, $D$ or $E$ and
the orbit categories:
\begin{itemize}
 \item[-] \emph{(Type $\rm{A}$)} $\db{A}{3n}/\tau^n[1]$, where $n\geq 1$;
 \item[-] \emph{(Type $\rm{D}$)} $\db{D}{kn}/(\tau\vph)^n$, where $n\geq 1$, $k>1$, $kn\geq 4$ and $\vph$ is induced by an automorphism
 of ${\rm D}_{kn}$ of order 2;
 \item[-] \emph{(Type $\rm{E}$)} $\db{E}{8}/\tau^4$ and $\db{E}{8}/\tau^8$.
\end{itemize}
\end{proposition}

\begin{proof}
These categories are described in a table of the appendix of
\cite{bikr}, and our description is based on that.
We explain why and how our description in case of types
  $\rm{A}$ and $\rm{D}_4$ differ from that of \cite{bikr}.

Apart from the cluster category,
the orbit categories of $\db{A}{m}$ appearing in the table of \cite{bikr} are given by the automorphisms:
\begin{itemize}
\item[] $(\tau^{\frac{m}{2}}[1])^{\frac{m+3}{3}}$, if 3 divides $m$
  and $m$ is even;
\item[]
\item[] $\tau^{\frac{m+3}{6}+\frac{m+1}{2}}[1]$, if 3 divides $m$ and
  $m$ is odd.
\end{itemize}

We simplify this description by using the fact that in the triangulated category
$\db{A}{m}$, we have \begin{equation} \label{eq:cy} \tau^{-(m+1)} =
  [2] \end{equation}
Note that this is sometimes referred to as a {\em fractional
Calabi--Yau property.}
 
Let $m=3n$. Assume first
that $n$ is even. We then have:
\begin{multline*}(\tau^{\frac{m}{2}}[1])^{\frac{m+3}{3}} =
(\tau^{3\frac{n}{2}}[1])^{n+1}
 = (\tau^{3n+1})^{\frac{n}{2}}\tau^n[n+1] =
(\tau^{3n+1})^{\frac{n}{2}}\tau^n[2]^\frac{n}{2} [1]  \\ =  (\tau^{3n+1}
[2])^{\frac{n}{2}}  ( \tau^n[1] ) = \tau^n[1] \end{multline*}
where the last equality follows from (\ref{eq:cy}).

Assume now that $n$ is odd. Then, we have
$$\tau^{\frac{m+3}{6}+\frac{m+1}{2}}[1] = \tau^{2n+1}[1] = (\tau^n[1])^{-1}$$
where (\ref{eq:cy}) is used for the last equation.

In both cases, the orbit category is $\db{A}{3n}/\tau^n[1]$.

The orbit categories of $\db{D}{4}$ appearing in the table of \cite{bikr}
are given by the automorphisms:
$\tau^k\sigma$, where $k$ divides 4, where $\sigma$ is induced by an automorphism
of ${\rm D}_4$ satisfying $\sigma^\frac{4}{k} = 1$ and where $(k,\sigma)\neq (1,1)$.
We thus have:
\begin{itemize}
\item[] if $k=1$, then $\sigma$ is of order 2; 
\item[] if $k=2$, then $\sigma$ is either the identity or of order 2;
\item[] if $k=4$, then $\sigma$ is the identity and the orbit category is the cluster category of type $\rm{D}_4$.
\end{itemize}
We claim that if $k=2$ and $\sigma$ is of order 2, then the corresponding orbit category
has non-zero maximal rigid objects, but does not have cluster tilting objects. Let thus $\sigma$
be of order 2. 
By computing the Hom-hammocks in the Auslander--Reiten quiver:
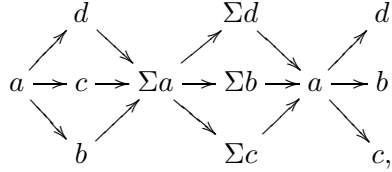
\[\xymatrix@-1pc{
& d \bdr & & \shift d \bdr & & d \\
a \dr\hdr\bdr & c \dr & \shift a \dr\hdr\bdr & \shift b \dr & a \dr\hdr\bdr & b \\
& b \hdr & & \shift c \hdr & & c,
}\]
one finds that $d$ and $\shift d$ are the only non-zero rigid objects and that
there are no non-zero morphisms from $d$ to $b$ or $c$. This shows that
$d$, and therefore also $\shift d$, are maximal rigid objects which are not cluster tilting.
This explains (L2). 
\end{proof}

\begin{proposition}[Amiot ; Burban--Iyama--Keller--Reiten]\label{proposition: list max rigid}
The standard, 2-Calabi--Yau, triangulated categories with finitely many indecomposable objects
and with non-zero maximal rigid objects which are not cluster tilting
are exactly the orbit categories:
\begin{itemize}
 \item[-] \emph{(Type A)} $\db{A}{(2t+1)(n+1)-3}/\tau^{t(n+1)-1}[1]$, where $n\geq 1$ and $t>1$;
 \item[-] \emph{(Type D)} $\db{D}{2t(n+1)}/\tau^{n+1}\vph^n$, where $n,t\geq 1$, and where $\vph$ is induced by an automorphism of
 ${\rm D}_{2t(n+1)}$ of order 2 ;
 \item[-] \emph{(Type E)} $\db{E}{7}/\tau^2$ and $\db{E}{7}/\tau^5$.
\end{itemize}
\end{proposition}

\begin{proof}
Type $\rm{A}$ deserves a few comments. The tables in the appendix of \cite{bikr} list
 all orbit categories of $\db{A}{m}$ with non-zero maximal rigid objects which are not cluster tilting. They are given
 by the following automorphisms:
\begin{itemize}
\item[I.] $(\tau^\frac{m}{2}[1])^k$, where $m$ is even; $k$ divides $m+3$; $k\neq 1$; $k\neq m+3$ and if 3 divides $m$, then $k\neq\frac{m+3}{3}$;
\item[] 
\item[II.] $\tau^{k + \frac{m+1}{2}}[1]$, where $m$ is odd;  $k$ divides $\frac{m+3}{2}$; $\frac{m+3}{2k}$ is odd;
 $k\neq \frac{m+3}{2}$ and if 3 divides $m$, then $k\neq\frac{m+3}{6}$
\end{itemize}
As in the proof of Proposition \ref{proposition: list CTO}, we use the property given by equation (\ref{eq:cy}),
in order to give a uniform description of all the cases above.
Note first that if $k=\frac{m+3}{3}$ or if $k=\frac{m+3}{6}$, then 3 divides $m$. Therefore, the condition
``if 3 divides $m$'' above is redundant.

Assume first we are in case I above, so $m$ is even and we can write $m+3 = uk$, where $u$ and $k$ are greater than 1 and $u\neq 3$.
We then have: 
\begin{multline*} (\tau^\frac{m}{2}[1])^k = (\tau^\frac{uk-3}{2}[1])^k
  = \tau^{k\frac{uk-3}{2}}[1][k-1] = \tau^{k\frac{uk-3}{2}}[1][2]^\frac{k-1}{2}\\
= \tau^{k\frac{uk-3}{2}}[1](\tau^{-uk+2})^\frac{k-1}{2}
= \tau^{\frac{u-1}{2}k-1}[1]\end{multline*}

Replacing $u$ by $2t+1$ and $k$ by $n+1$ gives $$m= uk-3 =
(2t+1)(n+1)-3 \text{ and } \frac{u-1}{2}k-1 = t(n+1) -1$$ Hence, we obtain the orbit categories
$\db{A}{(2t+1)(n+1)-3}/\tau^{t(n+1)-1}[1]$ (where $t>1$ and $n\geq 1$).

Assume now we are in case II, so that $m$ is odd and we can write $m+3 = 2uk$, where $u$ is odd and greater than 3.
We then have: 
$$\tau^{k + \frac{m+1}{2}}[1] = \tau^{k+uk-1}[1]
= \tau^{\frac{u+1}{2}2k-1}[1] = (\tau^{\frac{u-1}{2}2k-1}[1])^{-1}$$
where the last equation follows from equation (\ref{eq:cy}).

Replacing $u$ by $2t+1$, and $2k$ by $n+1$ gives $$m = 2uk-3=
(2t+1)(n+1)-3 \text{ and } \frac{u-1}{2}2k-1= t(n+1) -1$$ 
and also in this case we obtain the orbit categories
$\db{A}{(2t+1)(n+1)-3}/\tau^{t(n+1)-1}[1]$ (where $t>1$, $n\geq 1$).
\end{proof}

\begin{remark}
For a given value of $n$, the orbit categories \[\db{A}{(2t+1)(n+1)-3}/\tau^{t(n+1)-1}[1]\]
share some similarities, and are compared in section~\ref{section:
  comparisons}.
Note that when $t=1$ , we
  have $$\db{A}{(2t+1)(n+1)-3}/\tau^{t(n+1)-1}[1] =
  \db{A}{3n}/\tau^{n}[1].$$ Hence, by Proposition \ref{proposition: list CTO} 
this orbit category has cluster tilting
objects. On the other hand, if $t>1$ it has
non-zero maximal rigid objects which are not cluster tilting. This family can be expanded
by including the cluster tubes, thought of as a limit obtained when $t$ goes to infinity.
This point of view will be corroborated in sections \ref{section: comparisons} and \ref{section: endoalg},
where the endomorphism algebras of the maximal rigid objects in these categories are shown to be independent
of the specific value of $t$. 
\end{remark}

\subsection{The rigid objects}\label{subsection: rigid}

We will now describe indecomposable rigid objects in
the orbit categories listed in subsection \ref{subsection: list},
and then consider the additive subcategories generated by the set of rigid
objects.

\subsubsection{Type $\rm{A}$}\label{subsubsection: max rigid A}

In order to compute the rigid objects in the orbit categories
$\ac_{n,t} = \db{A}{(2t+1)(n+1)-3}/\tau^{t(n+1)-1}[1]$, we use the geometric description
\cite{ccs} of the cluster category of type $\rm{A}$.
The following lemma was implicitly used in the appendix of \cite{bikr}.

\begin{lemma}\label{lemma: rigid A}
 \begin{enumerate}
 \item There is a bijection between isomorphism classes of basic objects in $\ac_{n,t}$ and
collections of arcs of the $(2t+1)(n+1)$-gon which are stable under rotation by $\frac{2\pi}{2t+1}$.
Such a bijection is given in figure~\ref{figure: arcs type A17} for $t=2, n=3$ and is sketched in
figure~\ref{figure: Ant} for the general case.
 \item Under the bijection above, rigid objects correspond to non-crossing collections
of arcs. In particular:
\begin{enumerate}
 \item The isomorphism classes of  \sloppy indecomposable
  rigid objects in $\ac_{n,t}$ are parametrised by the arcs
$[i\;(i+2)],\ldots,[i\;(i+n+1)]$ for $i=1,\ldots,n+1$.
 \item  The maximal non-crossing collections correspond
to (isoclasses of) basic maximal rigid objects and such an object is cluster tilting if and only if the collection of arcs
is a triangulation (if and only if $t=1$).
\end{enumerate}
\end{enumerate}
\end{lemma}

\begin{center}
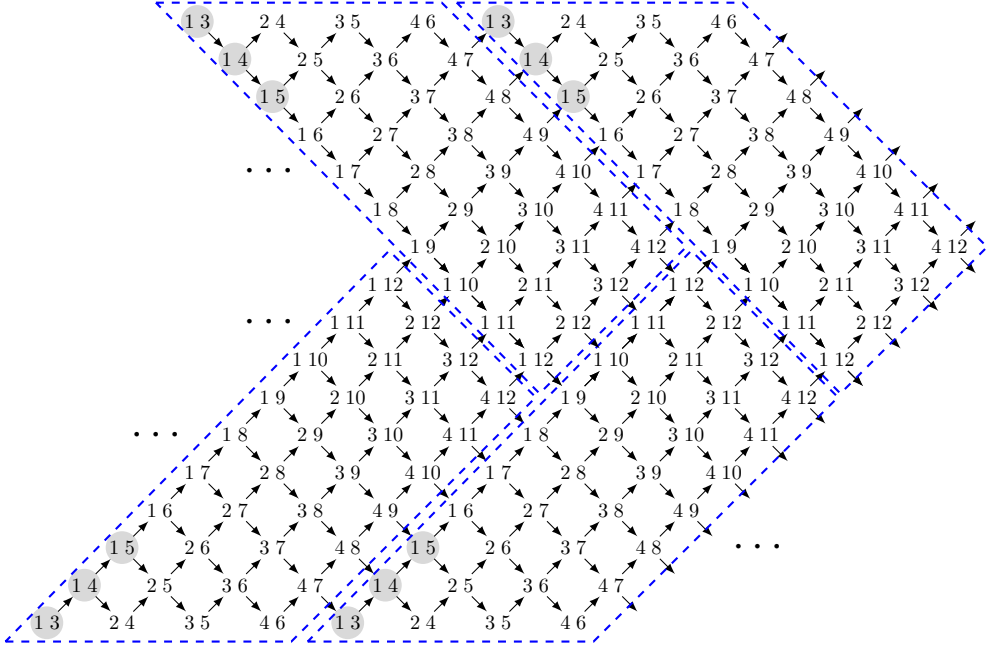
\begin{figure}
\begin{tikzpicture}[scale=0.5,
fl/.style={->,shorten <=6pt, shorten >=6pt,>=latex}]
\draw (0,0) node[circle, fill=black!15, scale=1.2] {} ;
\draw (1,1) node[circle, fill=black!15, scale=1.2] {} ;
\draw (2,2) node[circle, fill=black!15, scale=1.2] {} ;
\draw (8,0) node[circle, fill=black!15, scale=1.2] {} ;
\draw (9,1) node[circle, fill=black!15, scale=1.2] {} ;
\draw (10,2) node[circle, fill=black!15, scale=1.2] {} ;
\draw (4,16) node[circle, fill=black!15, scale=1.2] {} ;
\draw (5,15) node[circle, fill=black!15, scale=1.2] {} ;
\draw (6,14) node[circle, fill=black!15, scale=1.2] {} ;
\draw (12,16) node[circle, fill=black!15, scale=1.2] {} ;
\draw (13,15) node[circle, fill=black!15, scale=1.2] {} ;
\draw (14,14) node[circle, fill=black!15, scale=1.2] {} ;
\foreach \x in {1,2,3,4} {
 \pgfmathparse{12-\x}\let\z\pgfmathresult ;
  \foreach \y in {2,3,...,\z} {
   \newcount\u ;
   \pgfmathsetcount{\u}{\x+\y} ;
   \draw (2*\x+\y-4,\y-2) node[scale=0.7] {\x$\;$\the\u} ;
   \draw[fl] (\y-4+2*\x,\y-2) -- (\y-3+2*\x,\y-1) ;
   \draw[fl] (\y-3+2*\x,\y-1) -- (\y-2+2*\x,\y-2) ;
  } ;
} ;
\draw[thick, dashed, blue] (-1.1,-0.5) -- ++(7.6,0) -- ++(6.5,6.5) -- ++(-3.9,3.9) -- cycle ;
\begin{scope}[xshift=8cm]
 \foreach \x in {1,2,3,4} {
 \pgfmathparse{12-\x}\let\z\pgfmathresult ;
  \foreach \y in {2,3,...,\z} {
   \newcount\w ;
   \pgfmathsetcount{\w}{\x+\y} ;
   \draw (2*\x+\y-4,\y-2) node[scale=0.7] {\x$\;$\the\w} ;
   \draw[fl] (\y-4+2*\x,\y-2) -- (\y-3+2*\x,\y-1) ;
   \draw[fl] (\y-3+2*\x,\y-1) -- (\y-2+2*\x,\y-2) ;
  } ;
} ;
\end{scope}
\begin{scope}[xshift=240]
\draw[thick, dashed, blue] (-1.5,-0.5) -- ++(7.6,0) -- ++(6.5,6.5) -- ++(-3.9,3.9) -- cycle ;
\end{scope}
\begin{scope}[xshift=4cm, yshift=16cm, rotate=180, xscale=-1]
 \foreach \x in {1,2,3,4} {
 \pgfmathparse{12-\x}\let\z\pgfmathresult ;
  \foreach \y in {2,3,...,\z} {
   \newcount\r ;
   \pgfmathsetcount{\r}{\x+\y} ;
   \draw (2*\x+\y-4,\y-2) node[scale=0.7] {\x$\;$\the\r} ;
   \draw[fl] (\y-4+2*\x,\y-2) -- (\y-3+2*\x,\y-1) ;
   \draw[fl] (\y-3+2*\x,\y-1) -- (\y-2+2*\x,\y-2) ;
  } ;
} ;
\end{scope}
\begin{scope}[xshift=4.2cm, yshift=16.2cm, rotate=180, xscale=-1]
\draw[thick, dashed, blue] (-1.3,-0.3) -- ++(7.6,0) -- ++(6.5,6.5) -- ++(-3.9,3.9) -- cycle ;
\end{scope}
\begin{scope}[xshift=12cm, yshift=16cm, rotate=180, xscale=-1]
 \foreach \x in {1,2,3,4} {
 \pgfmathparse{12-\x}\let\z\pgfmathresult ;
  \foreach \y in {2,3,...,\z} {
   \newcount\r ;
   \pgfmathsetcount{\r}{\x+\y} ;
   \draw (2*\x+\y-4,\y-2) node[scale=0.7] {\x$\;$\the\r} ;
   \draw[fl] (\y-4+2*\x,\y-2) -- (\y-3+2*\x,\y-1) ;
   \draw[fl] (\y-3+2*\x,\y-1) -- (\y-2+2*\x,\y-2) ;
  } ;
 \draw (9.5-\x+2*\x,10.5-\x) node[circle, fill=white, scale=1.2] {} ;
} ;
\end{scope}
\begin{scope}[xshift=12.2cm, yshift=16.2cm, rotate=180, xscale=-1]
\draw[thick, dashed, blue] (-1.3,-0.3) -- ++(7.6,0) -- ++(6.5,6.5) -- ++(-3.9,3.9) -- cycle ;
\end{scope}
\draw (3,5) node[scale=1.5] {$\cdots$} ;
\draw (6,8) node[scale=1.5] {$\cdots$} ;
\draw (6,12) node[scale=1.5] {$\cdots$} ;
\draw (19,2) node[scale=1.5] {$\cdots$} ;
\end{tikzpicture}
\caption{A bijection between $\frac{2\pi}{5}$-periodic collections of arcs of the heptakaidecagon
and isomorphism classes of basic objects in $\ac_{3,2}$. The maximal rigid object of Corollary~\ref{corollary: max rigid A}
is highlighted in grey.}
\label{figure: arcs type A17}
\end{figure}
\end{center}

\begin{proof}
 Let $n,t\geq 1$, let $N = (2t+1)(n+1)$, let $\ac_{n,t}$ be the triangulated orbit category 
$\db{A}{N-3}/\tau^{t(n+1)-1}[1]$ and let $\cat_{\rm{A}_{N-3}}$ be the cluster category of type $\rm{A}_{N-3}$. Using
that $\ac_{n,t}$ is 2-CY, 
the universal property of orbit categories yields a functor
$\cat_{\rm{A}_{N-3}} \stackrel{F}{\gfl} \ac_{n,t}$. Note that this
covering functor commutes with shift functors
since these latter are induced by the shift in the orbit category $\db{A}{N-3}$.

In the cluster category $\cat_{\rm{A}_{N-3}}$, we have $\tau^{t(n+1)-1}[1] = \tau^{t(n+1)}$. Moreover,
in the derived category $\db{A}{N-3}$, we have $\tau^{N-2} = [-2]$. Therefore, $\tau$ is of order $N$ in
$\cat_{\rm{A}_{N-3}}$, and $\tau^{n+1}$ is of order $2t+1$. Since gcd$(t,2t+1) = 1$, then $\tau^{t(n+1)}$ is also of order
$2t+1$ and generates the same group as $\tau^{n+1}$.
The functor $F$ is thus a $(2t+1)$-covering functor, with $F(\tau^{n+1}X)$ isomorphic to $FX$ for any object $X$.
Since $F$ commutes with shifts, we have, for any two objects $X,Y$ in $\ac_{n,t}$:
$\Ext^1_{\ac_{n,t}}(X,Y) = \Hom_{\ac_{n,t}}(X,\shift Y) \simeq
\bigoplus_{FY'\simeq Y} \Hom_{\cat_{A_{N-3}}}(X,\shift Y')
= \bigoplus_{FY'\simeq Y} \Ext^1_{\cat_{A_{N-3}}}(X,Y')$.
We can thus use the description of the cluster category $\cat_{A_{N-3}}$ in terms of diagonals of the $N$-gon \cite{ccs}
in order to compute the rigid indecomposable objects in $\ac_{n,t}$: Isomorphism classes of indecomposable objects in $\ac_{n,t}$
are in bijection with collections of $2t+1$ diagonals of the $N$-gon which are stable under the automorphism sending a diagonal
$[i\;j]$ to $[(i+n+1)\;(j+n+1)]$. Moreover, such a collection corresponds to a rigid indecomposable object in $\ac_{n,t}$
if and only if none of its diagonals cross. This shows that isomorphism classes of indecomposable rigid objects
in $\ac_{n,t}$ are parametrised by the arcs $[i\;(i+2)],\ldots,[i\;(i+n+1)]$ for $i=1,\ldots,n+1$.

Consider a maximal collection $\mathfrak{A}$ of non-crossing arcs, stable under rotation by $\frac{2\pi}{2t+1}$,
that is not a triangulation. Then there exists an arc $\gamma$ which
does not cross any arc in the collection (such an arc will correspond
to a non-rigid indecomposable object). Necessarily, none of the rotations of
$\gamma$ by multiples of $\frac{2\pi}{2t+1}$ cross any arc in the collection.
This implies that the maximal rigid object corresponding to $\mathfrak{A}$ is not cluster tilting.
\end{proof}

\begin{center}
\begin{figure}
\begin{tikzpicture}[scale=0.475,
fl/.style={->,shorten <=5pt, shorten >=5pt,>=latex}]
\foreach \y in {0,1,2} {
\draw[fl] (1+\y,\y) -- (1+\y +1,\y +1) ;
\draw[fl] (2+\y,\y +1) -- (2+\y +1,\y ) ;
\draw (1+\y,\y) node[circle, fill=black!15, scale=1.2] {} ;
\newcount\u ;
\pgfmathsetcount{\u}{3+\y} ;
\draw (1+\y,\y) node[scale=0.7] {1$\;$\the\u} ;
} ;
\draw[fl] (3,0) -- (4,1) ;
\draw[fl] (4,1) -- (5,2) ;
\draw[fl] (4,1) -- (5,0) ;
\draw (6,0.5) node[scale=1.5] {$\cdots$} ;
\draw (7,2) node[scale=1.5] {$\cdots$} ;
\foreach \x in {1,2,3,4} {
\foreach \y in {0,1,2} {
\draw[fl] (5+2*\x +\y,\y) -- (5+2*\x +\y +1,\y +1) ;
\draw[fl] (5+2*\x +\y +1,\y +1) --(5+2*\x +\y +2,\y) ;
} ;
} ;
\foreach \y in {0,1,2} {
\draw[fl] (5+2*5 +\y,\y) -- (5+2*5 +\y +1,\y +1) ;
\draw (5+2*5 +\y,\y) node[circle, fill=black!15, scale=1.2] {} ;
\newcount\u ;
\pgfmathsetcount{\u}{3+\y} ;
\draw (5+2*5 +\y,\y) node[scale=0.7] {1$\;$\the\u} ;
} ;
\begin{scope}[xshift=-1cm, yshift=-1cm]
\draw (7,6) node[circle, fill=black!15, scale=1.2] {} ;
\draw (7,6) node[scale=0.7] {1$\; n\!+\!2$} ;
\draw[loosely dotted] (4.2,3.2) -- (6,5) ;
\draw[fl] (6,5) -- (7,6) ;
\draw[fl] (7,6) -- (8,5) ;
\draw[fl] (7,6) -- (8,7) ;
\draw[loosely dotted] (8,5) -- (9.8,3.2) ;
\draw[loosely dotted] (8.2,7.2) -- (15.8,14.8) ;
\draw[fl] (15,14) -- (16,15) ;
\draw[fl] (14,13) -- (15,14) ;
\draw (15,14) node[scale=0.7] {1\;$u$} ;
\draw[loosely dotted] (15.2,14.2) -- (16.8,15.8) ;

\draw[thick, dashed, blue] (0.5,0.5) -- (15,15) -- ++(7,-7) -- ++(-7.5,-7.5) --cycle ;
\draw[dashed, red] (5,6.5) -- ++(20,0) ;
\end{scope}

\begin{scope}[xshift=10cm, yshift=20cm, rotate=180, xscale=-1]
\foreach \y in {0,1,2} {
\draw[fl] (1+\y,\y) -- (1+\y +1,\y +1) ;
\draw[fl] (2+\y,\y +1) -- (2+\y +1,\y ) ;
\draw (1+\y,\y) node[circle, fill=black!15, scale=1.2] {} ;
\newcount\u ;
\pgfmathsetcount{\u}{3+\y} ;
\draw (1+\y,\y) node[scale=0.7] {1$\;$\the\u} ;
} ;
\draw[fl] (3,0) -- (4,1) ;
\draw[fl] (4,1) -- (5,2) ;
\draw[fl] (4,1) -- (5,0) ;
\draw (6,0.5) node[scale=1.5] {$\cdots$} ;
\draw (7,2) node[scale=1.5] {$\cdots$} ;
\foreach \x in {1,2,3,4} {
\foreach \y in {0,1,2} {
\draw[fl] (5+2*\x +\y,\y) -- (5+2*\x +\y +1,\y +1) ;
\draw[fl] (5+2*\x +\y +1,\y +1) --(5+2*\x +\y +2,\y) ;
} ;
} ;
\foreach \y in {0,1,2} {
\draw[fl] (5+2*5 +\y,\y) -- (5+2*5 +\y +1,\y +1) ;
\draw (5+2*5 +\y,\y) node[circle, fill=black!15, scale=1.2] {} ;
\newcount\u ;
\pgfmathsetcount{\u}{3+\y} ;
\draw (5+2*5 +\y,\y) node[scale=0.7] {1$\;$\the\u} ;
} ;

\begin{scope}[xshift=-1cm, yshift=-1cm]
\draw (7,6) node[circle, fill=black!15, scale=1.2] {} ;
\draw (7,6) node[scale=0.7] {1$\; n\!+\!2$} ;
\draw[loosely dotted] (4.2,3.2) -- (6,5) ;
\draw[fl] (6,5) -- (7,6) ;
\draw[fl] (7,6) -- (8,5) ;
\draw[fl] (7,6) -- (8,7) ;
\draw[loosely dotted] (8,5) -- (9.8,3.2) ;

\draw[loosely dotted] (16.2,3.2) -- ++(2.8,2.8) ;
\draw[dashed, red] (3,6.5) -- ++(17,0) ;
\end{scope}
\end{scope}
\draw[loosely dotted] (16.2,3.2) -- ++(11.8,11.8) ;
\draw[fl, green] (6.5,5) -- ++(9.8,9.8) ;
\draw (13,0) node[scale=.5] {$n\!+\!1\;n\!+\!3$} ;
\draw (14,1) node[scale=.5] {$n\!+\!1\;n\!+\!4$} ;
\draw (15,2) node[scale=.5] {$n\!+\!1\;n\!+\!5$} ;
\end{tikzpicture}
\caption{A fundamental domain for $\ac_{n,t}$ inside the derived category
is encircled by a dotted blue line. Below the bottom (and above the top)
dotted red line lie all rigid indecomposable objects.
The Hom-hammock of $(1\,n+2)$ is emphasized by a dotted rectangle.
The green arrow gives rise to a loop in the quiver $\qcc$.
Here $u$ equals $(t+1)(n+1)$.}
\label{figure: Ant}
\end{figure}
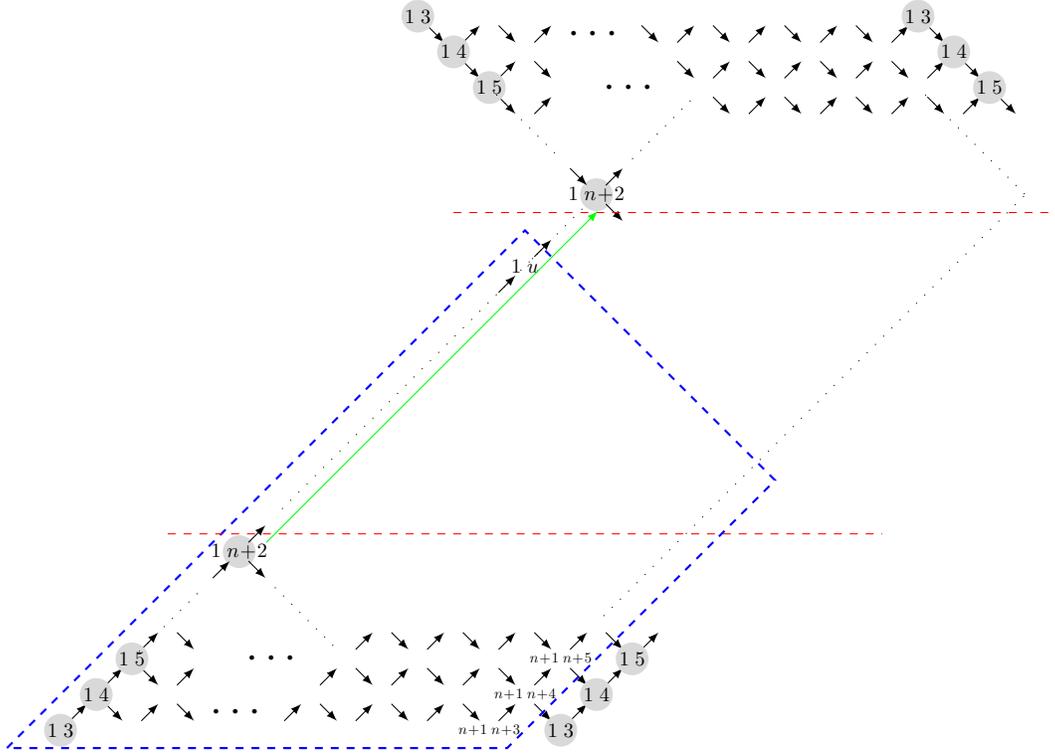
\end{center}

\begin{center}
\begin{figure}
\begin{tikzpicture}[scale=0.7]
\draw[very thick] (0,0) circle (5) ;
\foreach \a in {1,6,...,21} {
\draw (-14.4*\a+104.4:5) edge[very thick, color=black!80, out={-15.6-14.4*\a}, in={204.4-14.4*\a}] (-14.4*\a+75.6:5) ;
\draw (-14.4*\a+104.4:5) edge[very thick, color=black!60, out={-25.6-14.4*\a}, in={194.4-14.4*\a}] (-14.4*\a+61.2:5) ;
\draw (-14.4*\a+104.4:5) edge[very thick, color=black!40, out={-35.6-14.4*\a}, in={184.4-14.4*\a}] (-14.4*\a+46.8:5) ;
\draw (-14.4*\a+104.4:5) edge[very thick, color=black!20, out={-45.6-14.4*\a}, in={174.4-14.4*\a}] (-14.4*\a+32.4:5) ;
} ;
\foreach \x in {1,2,...,25} {
\draw (-14.4*\x+104.4:5) node {$\bullet$} ;
\draw (-14.4*\x+104.4:5.6) node {\x} ;
} ;
\end{tikzpicture}
\caption{A collection of arcs of the icosikaipentagon corresponding to a maximal rigid object in $\ac_{4,2}$.}
\label{figure: icosikaipentagon}
\end{figure}
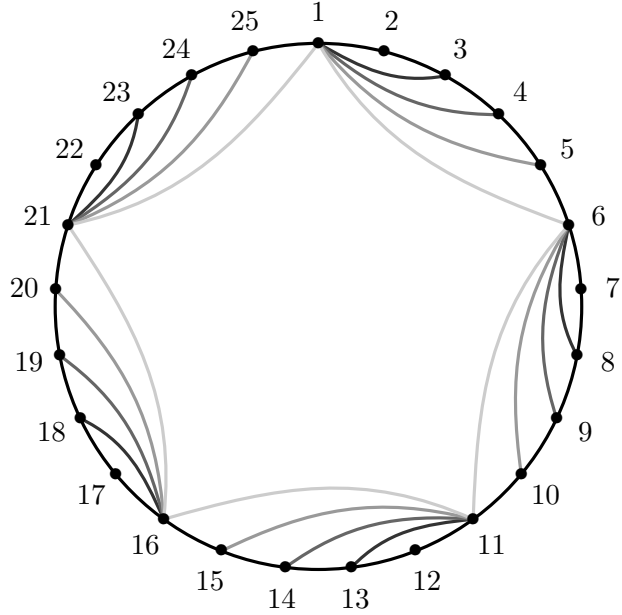
\end{center}

\begin{center}
\begin{figure}
\begin{tikzpicture}[scale=0.5]
\draw[very thick] (0,0) circle (5) ;
\foreach \a in {1,6,...,21} \draw[very thick, color=black!40] (-14.4*\a+104.4:5) -- (-14.4*\a+3.6:5) ;
\foreach \x in {1,2,...,25} {
\draw (-14.4*\x+104.4:5) node {$\bullet$} ;
\draw (-14.4*\x+104.4:5.7) node {\x} ;
} ;
\end{tikzpicture}
\caption{A collection of arcs of the icosikaipentagon corresponding to a non-rigid indecomposable object
of $\ac_{4,2}$.}
\label{figure: non-rigid type A}
\end{figure}
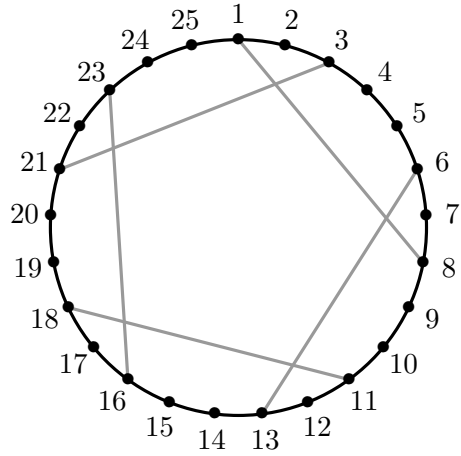
\end{center}

\begin{remark}
For an example of an arc corresponding to an indecomposable object which is not rigid, see figure~\ref{figure: non-rigid type A}.
\end{remark}

Let $\rc_{\ac_{n,t}}$ be the full additive subcategory of $\ac_{n,t}$
generated by the rigid objects.
We will show in section \ref{section: comparisons} that this category (up to equivalence)
only depends on $n$. Here we provide a first step towards that 
result. Recall that for an additive $\Hom$-finite 
Krull-Schmidt category $\uc$, the quiver $\qc_{\uc}$ of $\uc$ has 
vertices corresponding to the isomorphism classes of
indecomposable objects, and there are $\dim \irr(X,Y)$ arrows
from the vertex corresponding to $X$ to the 
vertex corresponding to $Y$, where $\irr(X,Y)$ is the 
space of irreducible maps from $X$ to $Y$.

\begin{proposition}\label{proposition: max rigid A}
The quiver $\qc_{\rc_{\ac_{n,t}}}\!\!\!\,$ is isomorphic to 
the quiver $\qc_n$ depicted in figure~\ref{figure: quiver}.
\end{proposition}

\begin{proof}
Consider the Auslander-Reiten quiver of $\ac_{n,t}$ depicted in
figure \ref{figure: Ant}.
Clearly, the
irreducible maps in $\ac_{n,t}$ with source and target in 
$\rc_{\ac_{n,t}}$, are also irreducible in $\rc_{\ac_{n,t}}$.  
It is also straightforward to verify,
by computations in the derived category 
$D^b(\field A_{(2t+1)(n+1)-3})$, that the map from $(1 \; n+2)$ 
to $(1 \; n+2)$ (and all shifts of this) is irreducible in $\rc_{\ac_{n,t}}$, and that there are no further irreducible maps
in $\rc_{\ac_{n,t}}$.
Hence the quiver $\qc_{\rc_{\ac_{n,t}}}$ is isomorphic to 
the quiver $\qc_n$ depicted in figure \ref{figure: quiver}.
\end{proof}

As a special case of the computations necessary for the proof of
Proposition \ref{proposition: max rigid A} we also obtain the 
following. Note that the cluster tilting case $t=1$ of this fact can also
be found in~\cite{bo}.

\begin{corollary}\label{corollary: max rigid A}
Let $n,t\in\nb$ and let $T$ be the maximal rigid object of the orbit category $\db{A}{(2t+1)(n+1)-3}/\tau^{t(n+1)-1}[1]$
corresponding to the collection of arcs generated by
$[1\;3],[1\;4],\ldots,[1\;n+2]$ (see Lemma~\ref{lemma: rigid A}). Then the
endomorphism algebra of $T$ is given by the quiver
\[
\xymatrix{ 1 \dr & 2 \dr & 3 \ar@{..}[r] & n-1 \dr & n \ar@(ur,dr)^{\al} },
\]
with ideal of relations generated by $\al^2$.
\end{corollary}

\begin{remark}
See figure \ref{figure: icosikaipentagon} for the 
collection of arcs corresponding to
the maximal rigid object in Corollary \ref{corollary: max rigid A}. 
\end{remark}

\begin{center}
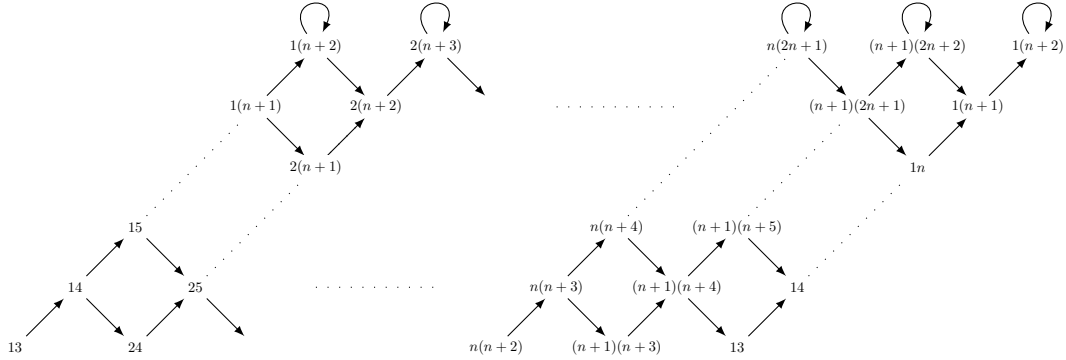
\begin{figure}
\begin{tikzpicture}[scale=0.8,
fl/.style={->,shorten <=6pt, shorten >=6pt,>=latex}]
\coordinate (13) at (0,0) ;
\coordinate (14) at (1,1) ;
\coordinate (15) at (2,2) ;
\coordinate (1n+1) at (4,4) ;
\coordinate (1n+2) at (5,5) ;
\coordinate (24) at (2,0) ;
\coordinate (25) at (3,1) ;
\coordinate (2n+1) at (5,3) ;
\coordinate (2n+2) at (6,4) ;
\coordinate (2n+3) at (7,5) ;
\coordinate (35) at (4,0) ;
\coordinate (3n+3) at (8,4) ;
\coordinate (13b) at (12,0) ;
\coordinate (14b) at (13,1) ;
\coordinate (1nb) at (15,3) ;
\coordinate (1n+1b) at (16,4) ;
\coordinate (1n+2b) at (17,5) ;
\coordinate (nn+2) at (8,0) ;
\coordinate (nn+3) at (9,1) ;
\coordinate (nn+4) at (10,2) ;
\coordinate (n2n+1) at (13,5) ;
\coordinate (n+1n+3) at (10,0) ;
\coordinate (n+1n+4) at (11,1) ;
\coordinate (n+12n+1) at (14,4) ;
\coordinate (n+12n+2) at (15,5) ;
\coordinate (13b) at (12,0) ;
\coordinate (14b) at (13,1) ;
\coordinate (1nb) at (15,3) ;
\coordinate (1n+1b) at (16,4) ;
\coordinate (n+1n+5) at (12,2) ;

\draw[fl] (nn+2) -- (nn+3) ;
\draw[fl] (nn+3) -- (n+1n+3) ;
\draw[fl] (nn+3) -- (nn+4) ;
\draw[fl] (nn+4) -- (n+1n+4) ;
\draw[fl] (n2n+1) -- (n+12n+1) ;
\draw[fl] (13) -- (14) ;
\draw[fl] (14) -- (15) ;
\draw[fl] (14) -- (24) ;
\draw[fl] (15) --(25) ;
\draw[fl] (1n+1) --(1n+2) ; 
\draw[fl] (1n+1) --(2n+1) ;
\draw[fl] (1n+2) --(2n+2) ; 
\draw[fl] (24) --(25) ;
\draw[fl] (25) --(35) ;
\draw[fl] (2n+1) --(2n+2) ;
\draw[fl] (2n+2) --(2n+3) ; 
\draw[fl] (2n+3) --(3n+3); 
\draw[fl] (n+1n+3) --(n+1n+4) ;
\draw[fl] (n+1n+4) --(13b) ;
\draw[fl] (n+12n+1) --(n+12n+2) ; 
\draw[fl] (n+12n+2) --(1n+1b) ; 
\draw[fl] (13b) --(14b) ;
\draw[fl] (1n+1b) --(1n+2b) ; 
\draw[fl] (n+12n+1) --(1nb) ;
\draw[fl] (1nb) --(1n+1b) ;
\draw[fl] (n+1n+4) --(n+1n+5) ;
\draw[fl] (n+1n+5) --(14b) ;

\draw (n2n+1) edge[out=125, in=65, loop, distance=1cm, fl]  (n2n+1) ;
\draw (1n+2) edge[out=125, in=65, loop, distance=1cm, fl]  (1n+2) ;
\draw (2n+3) edge[out=125, in=65, loop, distance=1cm, fl]  (2n+3) ; 
\draw (n+12n+2) edge[out=125, in=65, loop, distance=1cm, fl] (n+12n+2) ; 
\draw (1n+2b) edge[out=125, in=65, loop, distance=1cm, fl] (1n+2b) ; 
\draw[loosely dotted, shorten <=6pt, shorten >=6pt] (15) --(1n+1) ;
\draw[loosely dotted, shorten <=6pt, shorten >=6pt] (25) --(2n+2) ;
\draw[loosely dotted] (5,1) --(7,1) ;
\draw[loosely dotted] (9,4) --(11,4) ;
\draw[loosely dotted, shorten <=6pt, shorten >=6pt] (n+1n+5) --(n+12n+1) ;
\draw[loosely dotted, shorten <=6pt, shorten >=6pt] (14b) --(1nb) ;
\draw[loosely dotted, shorten <=6pt, shorten >=6pt] (nn+4) --(n2n+1) ;
\draw (13) node[scale=0.5] {13} ;
\draw (14) node[scale=0.5] {14} ;
\draw (15) node[scale=0.5] {15} ;
\draw (24) node[scale=0.5] {24} ;
\draw (25) node[scale=0.5] {25} ;
\draw (1n+1) node[scale=0.5] {$1(n+1)$} ;
\draw (1n+2) node[scale=0.5] {$1(n+2)$} ;
\draw (2n+1) node[scale=0.5, fill=white] {$2(n+1)$} ;
\draw (2n+2) node[scale=0.5] {$2(n+2)$} ;
\draw (2n+3) node[scale=0.5] {$2(n+3)$} ;
\draw (n+1n+3) node[scale=0.5] {$(n+1)(n+3)$} ;
\draw (n+1n+4) node[scale=0.5] {$(n+1)(n+4)$} ;
\draw (n+12n+1) node[scale=0.5] {$(n+1)(2n+1)$};
\draw (n+12n+2) node[scale=0.5] {$(n+1)(2n+2)$};
\draw (13b) node[scale=0.5] {13} ;
\draw (14b) node[scale=0.5] {14} ;
\draw (1nb) node[scale=0.5] {$1n$} ;
\draw (1n+1b) node[scale=0.5] {$1(n+1)$} ;
\draw (1n+2b) node[scale=0.5] {$1(n+2)$} ;
\draw (nn+3) node[scale=0.5] {$n(n+3)$} ;
\draw (nn+4) node[scale=0.5] {$n(n+4)$} ;
\draw (nn+2) node[scale=0.5] {$n(n+2)$};
\draw (n2n+1) node[scale=0.5] {$n(2n+1)$};
\draw (n+1n+5) node[scale=0.5] {$(n+1)(n+5)$};

\end{tikzpicture}
\caption{The quiver $\mathcal{Q}_n$}
\label{figure: quiver}
\end{figure}
\end{center}

\subsubsection{Type $\rm{D}$}\label{subsubsection: max rigid D}

Let $n,t\geq 1$ and let $P_{n,t}$ be a once-punctured $2t(n+1)$-gon.
We denote by $\rho$ the automorphism on the tagged arcs (see \cite{s,fst})
obtained by rotating by $\frac{\pi}{t}$ and switching tags, as in figure~\ref{figure: non-rigid type D}.

Recall that $\dc_{n,t}$  is the orbit category $\db{D}{2t(n+1)}/\tau^{n+1}\vph^n$.

\begin{landscape}
\begin{center}
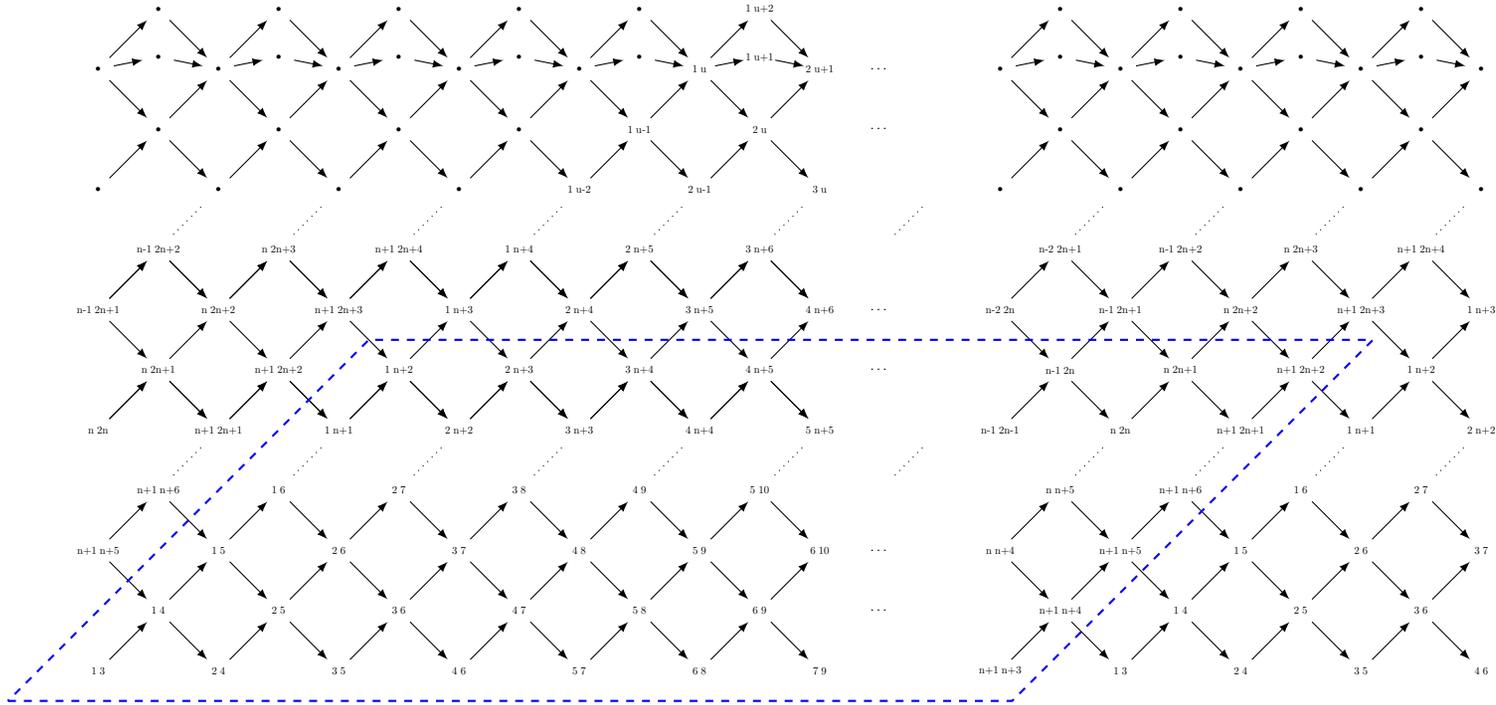
\begin{figure}
\begin{tikzpicture}[scale=0.80,
fl/.style={->,shorten <=6pt, shorten >=6pt,>=latex},
fli/.style={dotted,->,shorten <=8pt, shorten >=8pt,>=space}]
\newcommand\cs{0.37}
\foreach \x in {0,1,...,4} {
\foreach \y in {3,4,...,6} {
\newcount\r ;
\pgfmathsetcount{\r}{\x +\x + \y} ;
\pgfmathparse{int(\x+\y)}\let\z\pgfmathresult ;
\pgfmathparse{int(\x +1)}\let\h\pgfmathresult ;
\draw (\r -3 ,\y-3)  node[scale=\cs] {\h$\;$\z};
};
};

\foreach \x in {0,1,...,5} {
\draw[fl] (2*\x+1, 1) -- (2*\x + 2,2) ;
\draw[fl] (2*\x, 2) -- (2*\x +1,1) ;
\foreach \y in {0,1,...,3} { 
\draw[fl] (2*\x,2*\y) -- (2*\x + 1,2*\y+1) ;
\draw[fl] (2*\x +1,2*\y +1) -- (2*\x+2,2*\y) ;
};
};
\draw (13, 1)  node[scale=0.5] {$\cdots$};
\draw (13, 2)  node[scale=0.5] {$\cdots$};
\foreach \x in {0,1,...,3} {
\draw[fl] (2*\x+16, 1) -- (2*\x + 17,2) ;
\draw[fl] (2*\x +15, 2) -- (2*\x +16,1) ;
\foreach \y in {0,1} { 
\draw[fl] (2*\x +15,2*\y) -- (2*\x + 16,2*\y+1) ;
\draw[fl] (2*\x +16,2*\y +1) -- (2*\x+17,2*\y) ;
};
};
\foreach \x in {0,2,...,8,15,17,19,21,23} {
\draw (\x , 10)  node[scale=\cs] {$\bullet$};
};
\foreach \x in {1,3,...,9,16,18,20,22} {
\draw (\x , 10.2)  node[scale=\cs] {$\bullet$};
};
\foreach \x in {1,3,...,9,16,18,20,22} {
\draw (\x , 11)  node[scale=\cs] {$\bullet$};
};
\foreach \x in {1,3,...,7,16,18,20,22} {
\draw (\x , 9)  node[scale=\cs] {$\bullet$};
};
\foreach \x in {0,2,...,6,15,17,19,21,23} {
\draw (\x , 8)  node[scale=\cs] {$\bullet$};
};
\draw (0 , 2)  node[scale=\cs] {n+1$\;$n+5};
\draw (1 , 3)  node[scale=\cs] {n+1$\;$n+6};
\draw (10 , 0)  node[scale=\cs] {6$\;$8};
\draw (12 , 0)  node[scale=\cs] {7$\;$9};
\draw (11 , 1)  node[scale=\cs] {6$\;$9};
\draw (12 , 2)  node[scale=\cs] {6$\;$10};

\draw (15 , 0)  node[scale=\cs] {n+1$\;$n+3};
\draw (17 , 0)  node[scale=\cs] {1$\;$3};
\draw (19 , 0)  node[scale=\cs] {2$\;$4};
\draw (21 , 0)  node[scale=\cs] {3$\;$5};
\draw (23 , 0)  node[scale=\cs] {4$\;$6};
\draw (16 , 1)  node[scale=\cs] {n+1$\;$n+4};
\draw (18 , 1)  node[scale=\cs] {1$\;$4};
\draw (20,  1)  node[scale=\cs] {2$\;$5};
\draw (22 , 1)  node[scale=\cs] {3$\;$6};
\draw (15 , 2)  node[scale=\cs] {n$\;$n+4};
\draw (17 , 2)  node[scale=\cs] {n+1$\;$n+5};
\draw (19 , 2)  node[scale=\cs] {1$\;$5};
\draw (21 , 2)  node[scale=\cs] {2$\;$6};
\draw (23 , 2)  node[scale=\cs] {3$\;$7};
\draw (16 , 3)  node[scale=\cs] {n$\;$n+5};
\draw (18 , 3)  node[scale=\cs] {n+1$\;$n+6};
\draw (20 , 3)  node[scale=\cs] {1$\;$6};
\draw (22 , 3)  node[scale=\cs] {2$\;$7};

\draw (0 , 4)  node[scale=\cs] {n$\;$2n};
\draw (2 , 4)  node[scale=\cs] {n+1$\;$2n+1};
\draw (4 , 4)  node[scale=\cs] {1$\;$n+1};
\draw (6 , 4)  node[scale=\cs] {2$\;$n+2};
\draw (8 , 4)  node[scale=\cs] {3$\;$n+3};
\draw (10 , 4)  node[scale=\cs] {4$\;$n+4};
\draw (12 , 4)  node[scale=\cs] {5$\;$n+5};
\draw (1 , 5)  node[scale=\cs] {n$\;$2n+1};
\draw (3 , 5)  node[scale=\cs] {n+1$\;$2n+2};
\draw (5 , 5)  node[scale=\cs] {1$\;$n+2};
\draw (7 , 5)  node[scale=\cs] {2$\;$n+3};
\draw (9 , 5)  node[scale=\cs] {3$\;$n+4};
\draw (11 , 5)  node[scale=\cs] {4$\;$n+5};
\draw (0 , 6)  node[scale=\cs] {n-1$\;$2n+1};
\draw (2 , 6)  node[scale=\cs] {n$\;$2n+2};
\draw (4 , 6)  node[scale=\cs] {n+1$\;$2n+3};
\draw (6 , 6)  node[scale=\cs] {1$\;$n+3};
\draw (8 , 6)  node[scale=\cs] {2$\;$n+4};
\draw (10 , 6)  node[scale=\cs] {3$\;$n+5};
\draw (12 , 6)  node[scale=\cs] {4$\;$n+6};
\draw (1 , 7)  node[scale=\cs] {n-1$\;$2n+2};
\draw (3 , 7)  node[scale=\cs] {n$\;$2n+3};
\draw (5 , 7)  node[scale=\cs] {n+1$\;$2n+4};
\draw (7 , 7)  node[scale=\cs] {1$\;$n+4};
\draw (9 , 7)  node[scale=\cs] {2$\;$n+5};
\draw (11 , 7)  node[scale=\cs] {3$\;$n+6};

\draw (15 , 4)  node[scale=\cs] {n-1$\;$2n-1};
\draw (17 , 4)  node[scale=\cs] {n$\;$2n};
\draw (19 , 4)  node[scale=\cs] {n+1$\;$2n+1};
\draw (21, 4)  node[scale=\cs] {1$\;$n+1};
\draw (23, 4)  node[scale=\cs] {2$\;$n+2};
\draw (16 , 5)  node[scale=\cs] {n-1$\;$2n};
\draw (18 , 5)  node[scale=\cs] {n$\;$2n+1};
\draw (20 , 5)  node[scale=\cs] {n+1$\;$2n+2};
\draw (22 , 5)  node[scale=\cs] {1$\;$n+2};
\draw (15 , 6)  node[scale=\cs] {n-2$\;$2n};
\draw (17 , 6)  node[scale=\cs] {n-1$\;$2n+1};
\draw (19 , 6)  node[scale=\cs] {n$\;$2n+2};
\draw (21 , 6)  node[scale=\cs] {n+1$\;$2n+3};
\draw (23 , 6)  node[scale=\cs] {1$\;$n+3};
\draw (16 , 7)  node[scale=\cs] {n-2$\;$2n+1};
\draw (18 , 7)  node[scale=\cs] {n-1$\;$2n+2};
\draw (20, 7)  node[scale=\cs] {n$\;$2n+3};
\draw (22 , 7)  node[scale=\cs] {n+1$\;$2n+4};

\draw (8 , 8)  node[scale=\cs] {1$\;$u-2};
\draw (9 , 9)  node[scale=\cs] {1$\;$u-1};
\draw (10 , 10)  node[scale=\cs] {1$\;$u};
\draw (11 , 10.2)  node[scale=\cs] {1$\;$u+1};
\draw (11 , 11)  node[scale=\cs] {1$\;$u+2};
\draw (10 , 8)  node[scale=\cs] {2$\;$u-1};
\draw (11 , 9)  node[scale=\cs] {2$\;$u};
\draw (12 , 10)  node[scale=\cs] {2$\;$u+1};
\draw (12 , 8)  node[scale=\cs] {3$\;$u};

\foreach \x in {0,1,...,5} {
\draw[fl] (2*\x+1, 5) -- (2*\x + 2,6) ;
\draw[fl] (2*\x, 6) -- (2*\x +1,5) ;
\foreach \y in {2,3} { 
\draw[fl] (2*\x,2*\y) -- (2*\x + 1,2*\y+1) ;
\draw[fl] (2*\x +1,2*\y +1) -- (2*\x+2,2*\y) ;
};
};
\draw (13, 5)  node[scale=0.5] {$\cdots$};
\draw (13, 6)  node[scale=0.5] {$\cdots$};
\foreach \x in {0,1,...,3} {
\draw[fl] (2*\x+16, 5) -- (2*\x + 17,6) ;
\draw[fl] (2*\x +15, 6) -- (2*\x +16,5) ;
\foreach \y in {2,3} { 
\draw[fl] (2*\x +15,2*\y) -- (2*\x + 16,2*\y+1) ;
\draw[fl] (2*\x +16,2*\y +1) -- (2*\x+17,2*\y) ;
};
};

\foreach \x in {0,1,...,5} {
\draw[fl] (2*\x+1, 9) -- (2*\x + 2,10) ;
\draw[fl] (2*\x, 10) -- (2*\x +1,9) ;
\foreach \y in {4,5} { 
\draw[fl] (2*\x,2*\y) -- (2*\x + 1,2*\y+1) ;
\draw[fl] (2*\x +1,2*\y +1) -- (2*\x+2,2*\y) ;
};
};
\draw (13, 9)  node[scale=0.5] {$\cdots$};
\draw (13, 10)  node[scale=0.5] {$\cdots$};
\foreach \x in {0,1,...,3} {
\draw[fl] (2*\x+16, 9) -- (2*\x + 17,10) ;
\draw[fl] (2*\x +15, 10) -- (2*\x +16,9) ;
\foreach \y in {4,5} { 
\draw[fl] (2*\x +15,2*\y) -- (2*\x + 16,2*\y+1) ;
\draw[fl] (2*\x +16,2*\y +1) -- (2*\x+17,2*\y) ;
};
};

\foreach \x in {0,1,...,5} {
\draw[fl] (2*\x,10) -- (2*\x + 1,10.2) ;
\draw[fl] (2*\x+1,10.2) -- (2*\x + 2,10) ;
};

\foreach \x in {0,1,...,3} {
\draw[fl] (2*\x+15,10) -- (2*\x + 16,10.2) ;
\draw[fl] (2*\x+16,10.2) -- (2*\x + 17,10) ;
};

\foreach \x in {1,2,...,7} {
\foreach \y in {1,2} {
\draw[fli] (2*\x -1, 4*\y-1) -- (2*\x, 4*\y) ;
};
};

\foreach \x in {8,9,...,11} {
\foreach \y in {1,2} {
\draw[fli] (2*\x, 4*\y-1) -- (2*\x+1, 4*\y) ;
};
\label{key}};

\draw[thick, dashed, blue] (-1.5,-0.5) -- ++(16.7,0) -- ++(6,6) -- ++(-16.7,0) -- cycle ;




\end{tikzpicture}
\caption{The Auslander-Reiten quiver of $\dc_{n,t}$.
The objects in $\rc_{\dc_{n,t}}$ are in the
area inside the dashed blue lines. Here $u=2t(n+1)$.}
\label{figure: arcs type D}
\end{figure}
\end{center}
\end{landscape}

\begin{lemma}\label{lemma: rigid D}
\begin{enumerate}
 \item There is a bijection between isomorphism classes of basic objects in $\dc_{n,t}$ and
 collections of arcs of $P_{n,t}$ which are stable under $\rho$.
Such a bijection is illustrated in figure~\ref{figure: arcs type D}. 
 \item Under the above bijection, rigid objects correspond to non-crossing collections
 of arcs. In particular:
  \begin{enumerate}
   \item The isomorphism classes of indecomposable \sloppy rigid objects in $\dc_{n,t}$ are parametrised by the arcs
   $[i\;(i+2)],\ldots,[i\;(i+n+1)]$ for $i=1,\ldots,n+1$.
   \item  The maximal non-crossing collections which are stable under $\rho$
   correspond to (isoclasses of) basic maximal rigid objects.
  \end{enumerate}
\end{enumerate}
\end{lemma}

\begin{center}
\begin{figure}
\begin{tikzpicture}[scale=0.7]
\draw[very thick] (0,0) circle (3) ;
\foreach \x in {1,2,...,12} {
\draw (-30*\x+120:3) node {$\bullet$} ;
\draw (-30*\x+120:3.6) node {\x} ;
} ;
\draw (0,0) node {$\bullet$} ;
\draw (0,-0.5) node {$0$} ;
\draw[thick] (90:3) -- (0,0) node[midway, left, scale=0.85] {$\al$} ;
\draw[thick] (60:3) -- (0,0) node[midway, right, scale=0.85] {$\tau\al$} node[near end, sloped, rotate=90, scale=0.75] {$\bowtie$} ;
\begin{scope}[xshift=10cm]
\draw[very thick] (0,0) circle (3) ;
\foreach \a in {1,9,17}
\draw[thick, color=black!50] (-15*\a+90:3) -- (0,0) ;
\foreach \b in {5,13,21}
\draw[thick, color=black!50] (-15*\b+90:3) -- (0,0)
node[near end, sloped, rotate=90, scale=0.8] {$\bowtie$} ;
\draw (0,0) node {$\bullet$} ;
\foreach \x in {1,2,...,24} {
\draw (-15*\x+105:3) node {$\bullet$} ;
\draw (-15*\x+105:3.4) node[scale=0.8] {\x} ;
} ;
\end{scope}
\end{tikzpicture}
\caption{Action of $\tau$ on a tagged arc (left) and a non-rigid indecomposable object of $\dc_{3,3}$ (right).}
\label{figure: non-rigid type D}
\end{figure}
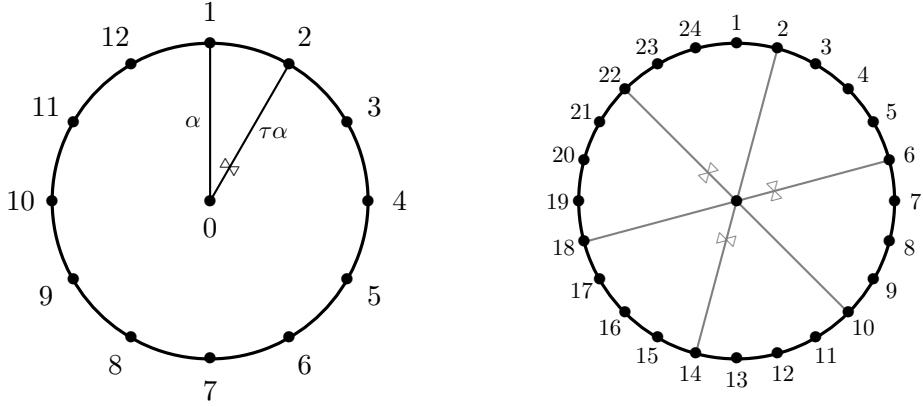
\end{center}

\begin{proof}
\sloppy The proof is similar to that of Lemma~\ref{lemma: rigid A}.
There is a $2t$-covering functor  from the cluster category $\cat_{{\rm D}_{2t(n+1)}}$ to the triangulated orbit category
$\dc_{t,n} = \db{D}{2t(n+1)}/\tau^{n+1}\vph^n$. We note that $\vph$ acts on arcs by switching tags and that
$\tau$ acts on arcs $[i\; 0]$ with an endpoint at the puncture 0 by
sending it to $[i+1\; 0]$ and by switching tags.
Therefore an arc with an endpoint at the puncture corresponds to a non-rigid indecomposable object in $\dc_{n,t}$ and
the rest of the proof is similar to that in type $\rm{A}$ above.
\end{proof}

Consider now the full additive subcategory $\rc_ {\dc_{n,t}}$, generated
by the rigid objects in $\dc_{n,t}$. We will show, in Section
\ref{section: comparisons}, that $\rc_ {\dc_{n,t}}$ is equivalent to $\rc_ {\ac_{n,t}}$.
For this, we will need the following.

\begin{proposition}\label{proposition: max rigid D}
The quiver $\qc_{\rc_{\dc_{n,t}}}$ is isomorphic to 
the quiver $\qc_n$ depicted in figure~\ref{figure: quiver}.
\end{proposition}

\begin{proof}

Consider the Auslander-Reiten quiver of $\dc_{n,t}$ depicted in
figure \ref{figure: arcs type D}.
Clearly, the
irreducible maps in $\ac_{n,t}$ with source and target in 
$\rc_{\dc_{n,t}}$, are also irreducible in $\rc_{\dc_{n,t}}$.  
To proceed, we will need some basic facts about Hom-hammocks in the derived category
$\db{D}{N}$, for $N$ even.
First note that, in the derived category $\db{D}{N}$, we have $\tau^{-N+1} = [1]$. Thus $\tau^{-N+2} = \tau [1]$
is a Serre functor in $\db{D}{N}$: For any $X,Y\in\db{D}{N}$, there are bi-natural isomorphisms
$\Hom_{\db{D}{N}}(X,Y) \simeq D\Hom_{\db{D}{N}}(Y,\tau^{-N+2}X)$. In particular, the Hom-hammock of any object
$X$ ends in $\tau^{-N+2}X$ and is symmetric with respect to the vertical line
(the blue line in figure~\ref{figure: Hom-hammocks}) going through $\tau^{-\frac{N}{2}+1}X$.
Without any computations, we thus obtain that the Hom-hammocks have the shape given in figure~\ref{figure: Hom-hammocks}, where
a part of the Hom-hammock of the indecomposable object denoted by $j$ is described.
The left-hand side of the figure is easily computed, since all meshes involved are commutative squares.
The rectangle on the left-hand side indicates some indecomposable objects $X$ such 
that $\dim \Hom (j,X) = 1$. Outside this rectangle, to its left and to its right, the zeros indicate 
that all morphisms from $j$ to some of the indecomposable objects in these regions are zero morphisms.
The star indicates a part of the Hom-hammock that we do not compute. The right-hand side of the figure
is deduced from the left-hand side by symmetry.
We have indicated some specific indecomposable objects in the figure. They are related by the following
equalities: $u = \tau^{-j+1}(1)$, $x = \tau^{-j+1}(N-j-1)$, $a=\tau^{-1}(x) = \tau^{-j}(N-j-1)$,
$b = \tau^{-N+2}(1)$, $c=\tau^{-N+2}(j)$ and $y = \tau^{-N+n+1}(n)$.

Using these Hom-hammocks, it is easy to verify that 
there is a non-zero map from $n$ to $y$, which
becomes an irreducible endomorphism in the category
$\rc_{\dc_{n,t}}$. For this,
note that there is a one-dimensional subspace of morphisms from $n$ to $y = \left(\tau^{n+1}\right)^{1-2t}(n)$
 (factoring through $N-1$) which do not factor through any indecomposable object in the $\tau$-orbit
 of $1,\ldots,n-1$. 
The same will obviously hold for the shifts of this
map. 

We claim that there are no other irreducible maps in 
$\rc_{\dc_{n,t}}$. This can be checked, using the 
Hom-hammocks of figure \ref{figure: quiver}.  
We leave the details to the reader, but 
point out the following useful fact.


Note that
the only indecomposable objects in the rectangles of figure~\ref{figure: Hom-hammocks}
that belong to the $\tau^{n+1}$ orbit of $1,2,\ldots,n$ are $1,2,\ldots,n$ and $y$. 
We claim that any morphism from some $j$, with $1\leq j\leq n$, to $y$ factors through $n$.
This holds since $\dim\Hom_{\db{D}{N}}(j,y) = 1$ and the composition $1 \fl 2 \fl \cdots \fl n \fl y$
 is non-zero (as can be seen from the case $j=1$ in figure~\ref{figure: Hom-hammocks}).
 
Hence the quiver $\qc_{\rc_{\dc_{n,t}}}$ is isomorphic to 
the quiver $\qc_n$ depicted in figure \ref{figure: quiver}.
\end{proof}

As for type $A$, we obtain the following as a special case of the computations necessary for the proof of
proposition \ref{proposition: max rigid D}.




\begin{center}
\begin{figure}
\begin{tikzpicture}[scale=0.4,
vertex/.style={fill=white, scale=0.85}]
\begin{scope}[y=-1cm]
\coordinate (1) at (0,0) ;
\coordinate (2) at (1,-1) ;
\coordinate (j) at (4,-4) ;
\coordinate (n) at (6,-6) ;
\coordinate (N-2) at (12,-12) ;
\coordinate (N-1) at (13,-13) ;
\coordinate (N) at (13,-12) ;
\coordinate (tau1) at (10,0) ;
\coordinate (x) at (16,-8) ;
\coordinate (y) at (20,-6) ;
\coordinate (u) at (8,0) ;
\coordinate (a) at (18,-8) ;
\coordinate (b) at (26,0) ;
\coordinate (c) at (30,-4) ;
\coordinate (d) at (22,-12) ;
\draw (1) --(2) --(j) ;
\begin{scope} [rotate=45]
\draw (j) rectangle (x) ;
\end{scope}
\draw (N-2) --(N-1) ;
\draw (N-2) --(N) ;
\draw (N) --(14,-12) ;
\draw (N-1) --(14,-12) ;
\draw (13,-11) --(14,-12) ;
\draw[loosely dashed] (14,-12) --(a) ;
\begin{scope} [rotate=45]
\draw (a) rectangle (c) ;
\end{scope}
\draw (20,-12) --(d) ;
\draw (20,-12) --(21,-11) ;
\draw (20,-12) --(21,-13) ;
\draw (21,-13) --(d) ;
\draw[loosely dotted] (j) --(c) ;
\draw[loosely dotted] (n) --(y) ;
\draw[blue] (17,0) -- (17,-14) ;
\draw (1) node {$\bullet$} node[above left] {$\scriptstyle{1}$} ;
\draw (2) node {$\bullet$} node[above left] {$\scriptstyle{2}$} ;
\draw (j) node[vertex] {$j$} ;
\draw (n) node[vertex] {$n$} ;
\draw (N-2) node {$\bullet$} node[above left] {$\scriptstyle{N-2}$} ;
\draw (N-1) node {$\bullet$} node[above left] {$\scriptstyle{N-1}$} ;
\draw (x) node[vertex] {$x$} ;
\draw (u) node[vertex] {$u$} ;
\draw (N) node {$\bullet$} ;
\draw (13,-11) node {$\bullet$} ;
\draw (14,-12) node {$\bullet$} ;
\draw (a) node[vertex] {$a$} ;
\draw (b) node[vertex] {$b$} ;
\draw (c) node[vertex] {$c$} ;
\draw (d) node {$\bullet$} ;
\draw (20,-12) node {$\bullet$} ;
\draw (21,-11) node {$\bullet$} ;
\draw (21,-13) node {$\bullet$} ;
\draw (21,-12) node {$\bullet$} ;
\draw (y) node[vertex] {y} ;
\draw (4,-1) node {\huge{0}} ;
\draw (4,-9) node {\huge{0}} ;
\draw (10,-6) node {\huge{1}} ;
\draw (14,-3) node {\huge{0}} ;
\draw (15.5,-12) node[scale=2] {$\ast$} ;
\draw (18.5,-12) node[scale=2] {$\ast$} ;
\draw (20,-3) node {\huge{0}} ;
\draw (24,-6) node {\huge{1}} ;
\draw (30,-1) node {\huge{0}} ;
\draw (30,-9) node {\huge{0}} ;
\end{scope}
\end{tikzpicture}
\caption{Hom-hammocks in the derived category $\db{D}{N}$, $N$ even.}
\label{figure: Hom-hammocks}
\end{figure}
\end{center}

\begin{corollary}\label{endo:D}
Let $n,t\in\nb$ and let $T$ be the maximal rigid object of the orbit category $\db{D}{2t(n+1)}/\tau^{n+1}\vph^n$
corresponding to the collection of arcs generated by
$[1\;3],[1\;4],\ldots,[1\;n]$ (see Lemma~\ref{lemma: rigid D}). Then the
endomorphism algebra of $T$ is given by the quiver
\[
\xymatrix{ 1 \dr & 2 \dr & 3 \ar@{..}[r] & n-1 \dr & n \ar@(ur,dr)^{\al} },
\]
with ideal of relations generated by $\al^2$.
\end{corollary}

\begin{proof}
The computation of the Gabriel quiver is
essentially included in the proof 
of Proposition \ref{proposition: max rigid D}. It is
easy to verify that the only relation is 
$\al^2$.
\end{proof}

Let $\Lambda_n$ denote the algebra appearing in 
Corollaries 
\ref{corollary: max rigid A} and \ref{endo:D}. We will need some 
properties of the module category $\module \Lambda_n$.
Recall that a module $M$ is called $\tau$-rigid if $\Hom(M, \tau M) =0$, see
\cite{air}.
Now let $\rc_{n}$ denote the full additive subcategory generated
by the indecomposable $\tau$-rigid modules in $\module \Lambda_n$.
It follows from Proposition \ref{proposition: list CTO}, with $t=1$, that
in particular $\Lambda_n$ is a 2-CY-tilted algebra, and so
by \cite{air}, a module is $\tau$-rigid if and only if it is of the form
$\Hom_{\C}(T,X)$, where $X$ is a rigid object in $\C= \db{A}{3n}/\tau^n[1]$.

It is easy to check that the quiver $\qc_{\rc_n}$ can be 
obtained by deleting the vertices labeled by $(n+1) (n+3), \dots ,
(n+1) (2n+2)$ in the quiver $\qc_n$ of figure \ref{figure: quiver}.

\subsubsection{Type $\rm{E}$}\label{subsubsection: max rigid E}

In this section we investigate the rigid (and maximal rigid) objects
in the orbit categories $\db{E}{7}/\tau^2$ and $\db{E}{7}/\tau^5$,
appearing in Proposition~\ref{proposition: list max rigid}.
There is also a geometric machinery available in type
 $\rm{E}$, see \cite{la}. However, our description instead relies on 
simple brute force computations, and we leave out almost all details.

For type  $\db{E}{7}/\tau^5$, the Auslander-Reiten quiver is given in Figure
\ref{figure: e7/5}.
There are 5 indecomposable rigid objects, all in the bottom $\tau$-orbit
in the figure. Let $x$ be any of these five. Then $x \oplus \tau^2 x$
is maximal rigid, and all maximal rigids are obtained this way.
In particular, they all have the same endomorphism ring.
  
\begin{proposition}\label{e7-alg}
The endomorphism algebra of any maximal rigid object in the orbit category
$\db{E}{7}/\tau^5$ is isomorphic to the path algebra of the quiver:
\[\xymatrix{\bullet \ar@(ul,dl)_{\alpha} \ar[r]^{\beta} & \bullet \ar@(ur,dr)^{\gamma}},\]
with ideal of relations generated by $\beta \al - \gamma \beta$, $\al^2$, $\gamma^2$.
\end{proposition}

\begin{remark}
This latter 2-endorigid algebra is shown not to be 
$2$-CY-tilted in section~\ref{subsection: not 2CY-tilted}.
\end{remark}

\begin{center}
\begin{figure}
\begin{tikzpicture}[scale=0.65,
fl/.style={->,shorten <=6pt, shorten >=6pt,>=latex}]

\foreach \x in {0,1,...,9} {
\draw (2*\x + 1, 2.2)  node[scale=0.5] {$\bullet$};
\draw[fl] (2*\x, 2) -- (2*\x + 1,2.2) ;
\draw[fl] (2*\x+1, 2.2) -- (2*\x + 2,2) ;
\foreach \y in {0,1,2} { 
\draw (2*\x +1, 2*\y +1)  node[scale=0.5] {$\bullet$};
\draw[fl] (2*\x,2*\y) -- (2*\x + 1,2*\y+1) ;
\draw[fl] (2*\x +1,2*\y +1) -- (2*\x+2,2*\y) ;
};
};

\foreach \x in {0,1,...,9} {
\foreach \y in {0,1} { 
\draw (2*\x, 2*\y+ 2)  node[scale=0.5] {$\bullet$};
\draw[fl] (2*\x,2*\y+2) -- (2*\x + 1,2*\y+1) ;
\draw[fl] (2*\x+1,2*\y+1) -- (2*\x + 2,2*\y+2) ;
};
};

\foreach \x in {0,1} {
\draw (10*\x, 0)  node[scale=0.7] {$a$};
\draw (10*\x+2, 0)  node[scale=0.7] {$b$};
\draw (10*\x+4, 0)  node[scale=0.7] {$c$};
\draw (10*\x+6, 0)  node[scale=0.7] {$d$};
\draw (10*\x+8, 0)  node[scale=0.7] {$e$};
};
\begin{scope}[xshift=0.0cm, yshift=0.0cm, rotate=0, xscale=-1]
\draw[thick, dashed, blue] (0.6,-0.3) -- ++ (-8.7,0) -- ++(-3.5, 2.8) -- ++
(-2.2, 2.9)
-- ++(8.7,0) -- cycle ;
\end{scope}

\end{tikzpicture}
\caption{The orbit category $\db{E}{7}/\tau^5$.}
\label{figure: e7/5}
\end{figure}
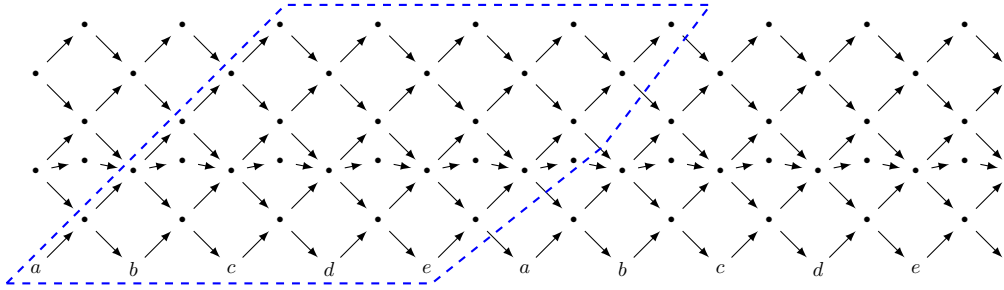
\end{center}

Let us now consider $\db{E}{7}/\tau^2$. 
Its Auslander-Reiten quiver is given in figure
\ref{figure: e7/2}.
There are only two indecomposable rigid objects, both in the top $\tau$-orbit
in the figure.

\begin{center}
\begin{figure}
\begin{tikzpicture}[scale=0.65,
fl/.style={->,shorten <=6pt, shorten >=6pt,>=latex}]

\foreach \x in {0,1,...,9} {
\draw (2*\x + 1, 2.2) node[scale=0.5] {$\bullet$};
\draw[fl] (2*\x, 2) -- (2*\x + 1,2.2) ;
\draw[fl] (2*\x+1, 2.2) -- (2*\x + 2,2) ;
\foreach \y in {0,1,2} { 
\draw[fl] (2*\x,2*\y) -- (2*\x + 1,2*\y+1) ;
\draw[fl] (2*\x +1,2*\y +1) -- (2*\x+2,2*\y) ;
};
};

\foreach \x in {0,1,...,9} {                          
\foreach \y in {0,1} {                                
\draw (2*\x +1, 2*\y +1) node[scale=0.5] {$\bullet$}; 
};
};

\foreach \x in {0,1,...,9} {
\foreach \y in {0,1} { 
\draw (2*\x, 2*\y+ 2) node[scale=0.5] {$\bullet$};
\draw[fl] (2*\x,2*\y+2) -- (2*\x + 1,2*\y+1) ;
\draw[fl] (2*\x+1,2*\y+1) -- (2*\x + 2,2*\y+2) ;
};
};

\foreach \x in {0,1,...,4} {
\draw (4*\x+1, 5)  node[scale=0.7] {$a$};           
\draw (4*\x+3, 5)  node[scale=0.7] {$b$};           
};

\foreach \x in {0,1,...,4} {                         
\draw (4*\x, 0)  node[scale=0.5] {$\bullet$};      
\draw (4*\x+2, 0)  node[scale=0.5] {$\bullet$};      
};
\begin{scope}[xshift=0.0cm, yshift=0.0cm, rotate=0, xscale=-1]
\draw[thick, dashed, blue] (0.6,-0.3) -- ++ (-2.7,0) -- ++(-3.5, 2.8) -- ++
(-2.2, 2.9)
-- ++(2.9,0) -- cycle ;
\end{scope}
\begin{scope}[xshift=4.05cm, yshift=0.0cm, rotate=0, xscale=-1]
\draw[thick, dashed, blue] (0.6,-0.3) -- ++ (-2.7,0) -- ++(-3.5, 2.8) -- ++
(-2.2, 2.9)
-- ++(2.9,0) -- cycle ;
\end{scope}
\end{tikzpicture}
\caption{The orbit category $\db{E}{7}/\tau^2$.}
\label{figure: e7/2}
\end{figure}
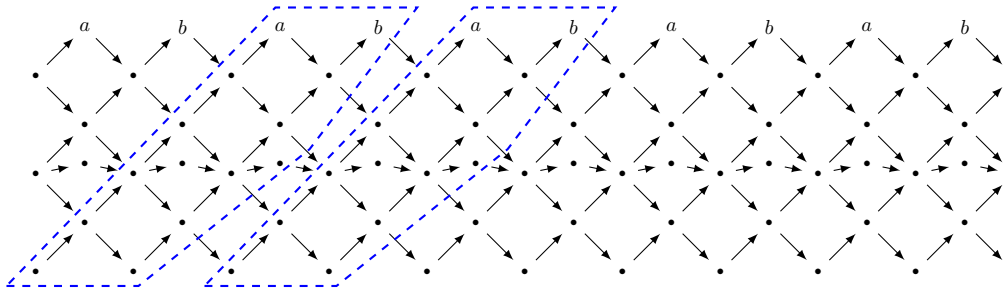
\end{center}

Now the full subcategory  $\rc_{\db{E}{7}/\tau^2}$
generated by the rigids only contains two indecomposable
objects with no maps between them.
In particular, we have the following. 

\begin{proposition}\label{prop:e7/2}
Any maximal rigid object in the orbit category $\db{E}{7}/\tau ^2$ is indecomposable
and its endomorphism algebra is given by a loop $\al$ with relation $\al^3$.
\end{proposition}

We can compare this to the case $\db{D}{4}/\tau \vph$,
which appears in Proposition
\ref{proposition: list CTO}.
The AR-quiver of $\db{D}{4}/\tau \vph$ is given by
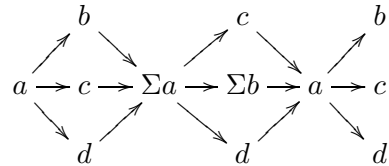
\[\xymatrix@-1pc{
& b \bdr & &  c \bdr & & b \\
a \dr\hdr\bdr & c \dr &  \shift a \dr\hdr\bdr &  \shift b \dr & a \dr\hdr\bdr & c \\
& d \hdr & &  d \hdr & & d
}\]
and the only indecomposable rigid objects are $b$ and $c$.
So, we also have that $\rc_{\db{D}{4}/\tau^2}$
contains exactly two indecomposable objects, 
with no maps between them. Moreover it is easily verified
that each of these indecomposables are maximal rigid,
and that the endomorphism rings are the same as in 
Proposition \ref{prop:e7/2}.

\subsection{Tables}\label{subsection: tables}

In a first table, we summarize some known results on orbit categories with cluster tilting objects,
which can be found in \cite{a,bikr,bo}.
A second table summarizes results from \cite{a,bikr} and from the current section. For each orbit category,
we give the number of isomorphism classes of indecomposable objects, the number of summands
of any basic maximal rigid object (or equivalently, the rank of the Grothendieck group of its
endomorphism algebra), the number of isomorphism classes of
indecomposable rigid objects,
and the quiver with relations of the endomorphism algebra of some maximal rigid object.
Recall that $\vph$ denotes an automorphism of the derived category of type ${\rm D}$ induced
by an automorphism of order two of a Dynkin diagram of type ${\rm D}$.

\begin{remark}
In the second row of Table 1, the following conventions are used:
\begin{itemize}
 \item If $n=1$, then $a=0$ and $b=0$;
 \item  if $k=2$, then there is no loop $\al$, and in the relations, $\al$ should be replaced by $ab$.
\end{itemize}
\end{remark}

\begin{remark}
Let $\cat$ be the orbit category appearing in the last row of the first table.
Because of the shape of the quiver in the last column, one might be tempted to think that $\cat$
should categorify a cluster algebra of type $\rm{F}_4$. However, $\cat$ has 24 indecomposable rigid objects only,
while there are 28 almost positive roots in type $\rm{F}_4$.

\end{remark}

\begin{landscape}\thispagestyle{empty}
\begin{figure}
\begin{tabular}{|c|c|c|c|c|c|}
\multicolumn{6}{c}{Table 1: Orbit categories with cluster tilting objects, which are not acyclic cluster categories.}\\
\multicolumn{6}{c}{}\\
\hline
\text{Orbit category}     &
\text{Indecomposables}    &
\text{Rank}               &
\text{Indec. rigids}       &
\text{Quiver}             &
\text{Relations}          \\
\hline
$\db{A}{3n}/\tau^n[1]
^{\phantom{\text{\huge{A}}}}
_{\phantom{\text{\huge{A}}}}$            &
$\frac{3n(n+1)}{2}$                      &
$n$                                      &
$n(n+1)$                                 &
$\xymatrix@C=1em@R=1em{ 1 \dr
& 2 \dr & 3 \ar@{..}[r]
& n-1 \dr & n \ar@(ur,dr)^{\al} }$       &
$\al^2$                                  \\
\hline
$\db{D}{kn}/\tau^n\vph^n$, $kn\geq 4$, $k>1^{\phantom{\text{\huge{A}}}}$                   &
$kn^2$                                          &
$n$                                             &
$n(n+1)$                                     &
$\xymatrix@C=1em@R=1em{ 1 \dr & 2 \dr & 3
\ar@{..}[r] & n-1 \dr^a & n \ar@/^/[l]^b
\ar@(ur,dr)^{\al}}$                             &
$\al^{k-1}-ab$, $\al a$, $b\al$                 \\
\hline
$\db{E}{8}/\tau^4
\phantom{\text{\huge{A}}}
_{\phantom{\text{\huge{A}}}}$  &
32                             &
2                              &
8                              &
$\xymatrix@C=1em@R=1em{
1 \dr &
2 \ar@(ur,dr)^{\al}}$          &
$\al^3$                        \\
\hline
$\db{E}{8}/\tau^8
\phantom{\text{\huge{A}}}
_{\phantom{\text{\huge{A}}}}$     &
64                                &
4                                 &
24                                &
$\xymatrix@C=1em@R=1em{
1 \dr & 2 \dr^a &
3 \ar@/^/[l]^b \dr & 4}$          &
$aba$, $bab$                      \\
\hline
\end{tabular}

\vspace{2cm}

\begin{tabular}{|c|c|c|c|c|c|}
\multicolumn{6}{c}{Table 2: Orbit categories with non-cluster tilting, maximal rigid objects.}\\
\multicolumn{6}{c}{}\\
\hline
\text{Orbit category}     &
\text{Indecomposables}    &
\text{Rank}               &
\text{Indec. rigids}       &
\text{Quiver}             &
\text{Relations}          \\
\hline
$\db{A}{(2t+1)(n+1)-3}/\tau^{k(n+1)-1}[1]\phantom{^\text{\huge{A}}}$ &
$\scriptstyle{\frac{1}{2}[(2t+1)(n+1)-3](n+1)}$           &
$n$                                                       &
$n(n+1)$                                                  &
$\xymatrix@C=1em@R=1em{ 1 \dr & 2 \dr & 3 \ar@{..}[r]
& n-1 \dr & n \ar@(ur,dr)^{\al} }$                        &
$\al^2$                                                   \\
$t>1$ &
& & & & \\
\hline
$\db{D}{2t(n+1)}/\tau^{n+1}\vph^n
\phantom{\text{\huge{A}}}
_{\phantom{\text{\huge{A}}}}$            &
$2t(n+1)^2$                              &
$n$                                      &
$n(n+1)$                                 &
$\xymatrix@C=1em@R=1em{ 1 \dr
& 2 \dr & 3 \ar@{..}[r]
& n-1 \dr & n \ar@(ur,dr)^{\al} }$       &
$\al^2$                                  \\
\hline
$\db{E}{7}/\tau^2
\phantom{\text{\huge{A}}}
_{\phantom{\text{\huge{A}}}}$       &
14                                  &
1                                   &
2                                   &
$\xymatrix{1
\ar@(ur,dr)^{\al}}$                 &
$\al^3$                             \\
\hline
$\db{E}{7}/\tau^5
\phantom{\text{\huge{A}}}
_{\phantom{\text{\huge{A}}}}$    &
35                               &
2                                &
5                                &
$\xymatrix@C=1em@R=1em{1
\ar@(ul,dl)_{\alpha}
\ar[r]^{\beta} & 2
\ar@(ur,dr)^{\gamma}}$           &
$\beta \al - \gamma \beta$,
$\al^2$, $\gamma^2$              \\
\hline
\end{tabular}
\end{figure}

\end{landscape}

\section{Comparing subcategories generated by rigid objects}\label{section: comparisons}

Our aim, in this section, is to compare the full subcategories of rigid objects of the triangulated
categories listed in Table 2.
In order to do so, we will follow  a strategy we now describe: Let $\cat$ and $\dc$
be $\field$-linear, Krull--Schmidt, Hom-finite, 2-Calabi--Yau, triangulated categories. We assume that $T\in\cat$ is a cluster tilting object
and $U\in\dc$ a maximal rigid object.
Let $\rc_\cat$, resp. $\rc_\dc$, be the full subcategory of $\cat$,
resp. $\dc$, generated by the rigid objects.
Let $\qcc$ be a quiver whose vertices are the (isoclasses of) indecomposable rigid objects of $\cat$,
and whose arrows form a basis for the irreducible morphisms in $\rc_\cat$. Define
$\qcd$ similarly. Finally, let $\qc^{\tau-\text{rig}}_\cat$ be the quiver similarly given
by the irreducible morphisms of the image of $\cat(T,-)|_{\rc_\cat}$ in $\module\End_\cat(T)$.
Define $\qc^{\tau-\text{rig}}_\dc$ similarly.

\vspace{5pt}
Assume that the following hold:
\begin{itemize}
 \item[(a)] The indecomposable rigid objects of $\cat$ are all shifts of
indecomposable summands of $T$; and similarly for $\dc$.
 \item[(b)] There is some isomorphism of quivers $\sigma: \qcc \fl \qcd$
satisfying the following properties:
\begin{itemize}
 \item[(b1)] The map $\sigma$ commutes with shifts on objects and on irreducible morphisms;
 \item[(b2)] It sends $T$ to $U$;
 \item[(b3)] It induces an isomorphism between $\End_\cat(T)$ and $\End_\dc(U)$.
\end{itemize}
 \item[(c)] The finite dimensional algebra $\End_\cat(T)$ is generalised standard, i.e. the morphisms in the module category are given by linear combinations of paths in its Auslander--Reiten quiver \cite{sko}.
 \item[(d)] The quiver $\qc^{\tau-\text{rig}}_\cat$ is isomorphic to the full subquiver
 of $\qcc$ whose vertices are not in $\add \shift T$; and similarly for $\dc$.
\end{itemize}

\begin{lemma}\label{lemma: comparing subcategories of rigids}
 Under the assumptions listed above, any morphism in $\rc_\cat$ is
a linear combination of paths in $\qcc$ and $\sigma$ induces an equivalence of categories
$\rc_\cat\fl\rc_\dc$.
\end{lemma}

\begin{proof}
Assume that $T=T_1\oplus\cdots\oplus T_n$ is basic, and $T_i$ is indecomposable for each $i$.


We prove the statement in three steps:
\begin{enumerate}
 \item  Any morphism in $\rc_\cat$ (resp. $\rc_\dc$) is
a linear combination of paths in $\qcc$ (resp. $\qcd$).
 \item The morphism $\sigma$ induces a well-defined functor $\rc_\cat \fl \rc_\dc$,
which is faithful.
 \item The induced functor is dense
and full.
\end{enumerate}

(1) Let $f$ be a morphism in $\rc_\cat$. By assumption (a), we may assume that it is of the form
$\shift^a T_i \fl \shift^b T_j$ for some $a,b\in\zb$, and $i,j\in\{1,\ldots,n\}$.
By assumption (c), the morphism $\cat(T,\shift^{-a}f)$ is a linear combination of paths in $\qcc^{\tau-\text{rig}}$.
Let $g\in\rc_\cat$ be the corresponding linear combination of paths in $\qcc$. Such a morphism exists by assumption (d).
We then have $\cat(T, \shift^{-a}f-g) = 0$ so that $\shift^{-a}f-g$ belongs to the ideal $(\shift T)$.
Since the domain of $\shift^{-a}f$ lies in $\add T$, and $T$ is rigid,
we have $\shift^{-a}f=g$ and $f$ is a linear combination of paths in $\qcc$.

(2) Let $f$ be a linear combination of paths in $\qcc$. We claim that $f=0$ in $\rc_\cat$
if and only if $\sigma f = 0$ in $\rc_\dc$. Indeed:
\begin{eqnarray*}
 f = 0 \text{ in } \rc_\cat & \Leftrightarrow & \shift^{-a} f = 0 \text{ in } \rc_\cat \\
 & \Leftrightarrow & \cat(T,\shift^{-a} f) = 0 \text{ in } \module\End_\cat(T) \\
 & \Leftrightarrow & \dc(\sigma T, \sigma \shift^{-a}f) = 0 \text{ in } \module\End_\dc(U) \\
 & \Leftrightarrow & \dc(U,\shift^{-a}\sigma f) = 0 \text{ in } \module\End_\dc(U) \\
 & \Leftrightarrow & \shift^{-a}\sigma f = 0 \text{ in } \rc_\dc \\
 & \Leftrightarrow & \sigma f = 0 \text{ in } \rc_\dc.
\end{eqnarray*}
The second equivalence uses the fact that the domain of $\shift^{-a}f$ belongs to $\add T$,
the fourth equivalence follows from assumptions (b1) and (b2). The third equivalence follows from assumption (d)
as follows: This assumption implies that $\sigma$ induces an isomorphism from $\qcc^{\tau-\text{rig}}$
to $\qcd^{\tau-\text{rig}}$ which commutes with the inclusions into $\qcc$ and $\qcd$.

(3) By construction, the functor $\rc_\cat\fl\rc_\dc$ induced by $\sigma$ is dense.
 For all $i=1,\ldots, n$, let $U_i$ be $\sigma T_i$. Let $g$ be a morphism in $\rc_\dc$.
As above, we may assume that it is of the form $U_i \fl \shift^k U_j$,
for some $k\in\zb$ and some $i,j\in\{1,\ldots,n\}$.
There is some $f\in\cat(T_i,\shift^k T_j)$ whose image in $\module\End_\cat(T)$ is associated with
$\dc(U,g)$ in $\module\End_\dc(U)$. We thus have $\dc(U, g- \sigma f) = 0$, which implies $\sigma f = g$.
The functor induced by $\sigma$ is full.
\end{proof}

\begin{proposition}\label{prop: equivalences}
 For all $t\geq 1$, there are equivalences of additive categories:
\begin{enumerate}
 \item $\rc_{\ac_{n,t}} \simeq \rc_{\ac_{n,1}}$;
 \item $\rc_{\dc_{n,t}} \simeq \rc_{\ac_{n,1}}$;
 \item $\rc_{E_{7,2}} \simeq \rc_{\dc_{4, \tau \varphi}}$;
\end{enumerate}
\end{proposition}

\begin{proof}
For each case, we need to check the assumptions of 
Lemma~\ref{lemma: comparing subcategories of rigids}. This is done in Sections \ref{subsubsection: max rigid A}, \ref{subsubsection: max rigid D} and
\ref{subsubsection: max rigid E}, respectively. 
\end{proof}

\section{2-endorigid algebras of finte type}\label{section: endoalg}

\subsection{A 2-endorigid algebra which is not 2-CY tilted}\label{subsection: not 2CY-tilted}
Consider the algebra $\Gamma =\field Q/I$, where $Q$ is the quiver
\bigskip
$$\xymatrix@C=0.3cm@R=0.1cm{
1  \ar@(ul,dl)_{\alpha} \ar[rrr]^{\beta} & && 2 \ar@(ur,dr)^{\gamma}}
$$
\bigskip
\bigskip
and the relations are $\beta \alpha - \gamma\beta, \alpha^2, \gamma^2 $.

The indecomposable projectives in $\module \Gamma$ are given by
$$P_1 = \begin{pmatrix}  & 1 & \\ 1& & 2 \\ &2& \end{pmatrix} \text{    and   } 
P_2 = \begin{pmatrix}   2 \\ 2 \end{pmatrix}, $$

while the indecomposable injectives are 
$$I_1 = \begin{pmatrix}  1  \\ 1 \end{pmatrix} \text{    and   } 
I_2 =  P_1. $$

We have a minimal injective coresolution of $\Gamma = P_1 \amalg P_2$ given by
$$
0 \to P_1 \amalg P_2 \to I_2 \amalg I_2 \to I_1 \to 0
$$
and hence $\id \Gamma = 1$, that is, $\Gamma$ is Gorenstein of dimension 1.
Then, see \cite{kr}, we have that $\Sub \Gamma$ is a Frobenius category with projective (=injective)
objects $\add \Gamma$, and 
$\underline{\Sub} \Gamma$ is a triangulated category, with
suspension functor isomorphic to $\Omega_{\Sub \Gamma}^{-1}$. 

We claim that $\Gamma$ is not a 2-CY-tilted algebra. To see this, consider 
the simple $S_2$ and the module $X=  \begin{pmatrix}   1 \\ 2 \end{pmatrix}$.
The exact sequence
$$0 \to S_2 \to P_2 \to S_2 \to 0$$ 
in $\Sub \Gamma$, shows that $\Omega^{-1} (S_2) \simeq S_2 
\simeq \Omega^{1}(S_2)$.
Hence also $\Omega^{-3} (S_2) \simeq S_2$.
We then have that $\underline{\Hom}(S_2, X) \neq 0$, while clearly $\underline{\Hom}(X,S_2) = 0$.
Therefore $\underline{\Sub} \Gamma$ is not 3-Calabi-Yau, and this implies that $\Gamma$ is
not a 2-CY-tilted algebra, by \cite{kr}.
The same argument shows that $\Gamma$ is not $d$-CY-tilted for $d\geq 2$.

\subsection{Standard 2-Calabi--Yau categories}

Recall that our base field $\field$ is assumed to be algebraically closed and of characteristic 0.
In that setup, it is known from \cite{bgrs} that all finite-dimensional algebras of finite representation type
are standard: Their module categories are the path categories on their Auslander--Reiten quivers
modulo all mesh relations. In this section, we adress the following related question: Let $\cat$
be a triangulated category of finite type. If $\cat$ is 2-CY with cluster tilting objects,
is it standard? We were not able to answer this question so far.
However, we prove here that $\cat$ is generalised standard 
\cite{sko} in the following sense. 

\begin{definition} \rm
A $\field$-linear, Krull--Schmidt, Hom-finite, triangulated category with a Serre functor is called \emph{generalised standard}
if all of its morphisms are given by linear combinations of paths in its Auslander--Reiten quiver.
\end{definition}

\begin{proposition}\label{proposition: generalised standard}
 Let $\cat$ be a $\field$-linear, Krull--Schmidt, 2-Calabi--Yau, triangulated category.
Assume that $T\in\cat$ is a cluster tilting object whose endomorphism algebra
is generalised standard. Then $\cat$ is generalised standard.
\end{proposition}

\begin{proof}
 Let $\Gamma$ be the Auslander--Reiten quiver of $\cat$, and
$\overline{\Gamma}$ be the one of $\Endt$. By~\cite[Proposition 3.2]{bmr}
the AR-sequences in $\modt$ are induced by the AR-triangles in $\cat$. 
It follows that $\overline{\Gamma}$ is naturally a full subquiver of
$\Gamma$
and that we can pick a basis $(e_\al)_{\al\in\Gamma_1}$ of irreducible morphisms in $\cat$
adapted to $\Gamma$ (i.e. satisfying the mesh relations)
such that $(\cat(T,e_\al))_{\al\in\overline{\Gamma}_1}$ is a basis
of irreducible morphisms in $\modt$ adapted to $\overline{\Gamma}$. In what follows,
we will use the following notation: if $p=\sum_i \ld_i \al^i_{k_i}\cdots\al^i_1$ is a
linear combination of paths in $\Gamma$, we write $e_p$ for the morphism
$\sum_i\ld_i e_{\al^i_{k_i}}\circ\cdots\circ e_{\al^i_1}$.
We note that the statement of the lemma is an immediate consequence of the two claims below.

\emph{Claim} 1: Any morphism $f$ in $\cat$ is of the form $f = e_p+g$,
where $p$ is a linear combination of paths in $\Gamma$ and
$g$ belongs to the ideal $(\shift T)$.

\emph{Proof of Claim} 1: Since $\Endt$ is generalised standard, $\cat(T,f)$ is of the form $\cat(T,e_p)$
where $p$ is a linear combination of paths in $\overline{\Gamma}$, viewed as
a subquiver of $\Gamma$. We thus have $f = e_p + g$ for some $g\in(\shift T)$.

\emph{Claim} 2: Any morphism $g\in(\shift T)$ is of the form $e_p$, for some
linear combination $p$ of paths in $\Gamma$.

\emph{Proof of Claim} 2: Let $X\stackrel{g}{\gfl}Y$ belong to $(\shift T)$. Then there are some $U\in\add T$,
$\shift U\stackrel{a}{\gfl}Y$ and $X\stackrel{b}{\gfl}\shift U$ such that $g = ab$.
Applying Claim 1 to $\shift b$ gives a linear combination $q$ of paths in $\Gamma$
and a morphism $h$ in $(\shift T)$ such that $\shift b = e_q + h$. Since $T$ is rigid
and $\shift b$ has codomain in $\add\shift^2T$, then $h$ is zero. A similar argument
shows that $\susm a$ is of the form $e_r$. The claim follows.
\end{proof}

\begin{corollary}\label{corollary: generalised standard}
Let $\cat$ be a $\field$-linear, Krull--Schmidt, 2-Calabi--Yau, triangulated category.
Assume that $T\in\cat$ is a cluster tilting object whose endomorphism algebra is of finite representation type.
Then $\cat$ is generalised standard.
\end{corollary}


\subsection{The standard 2-endorigid algebras of finite representation type}

We call a finite dimensional $\field$-algebra \emph{standard 2-endorigid} if it is isomorphic to the endomorphism algebra
of a maximal rigid object in a standard, ($\field$-linear, Krull--Schmidt) $2$-Calabi--Yau,
triangulated category. 

The standard 2-CY-tilted algebras of finite representation type were classified by
Bertani--Oppermann in \cite{bo}, where a quiver with potential
is given for each isomorphism class. Ladkani noticed, see \cite{l}, that a 2-CY category with cluster tilting objects was missing
in the list given in \cite[Appendix]{bikr}. For a comprehensive classification of all
standard 2-CY-tilted algebras of finite representation type one thus has to take the algebra appearing in \cite{l} into account.

\begin{theorem}
The connected, standard 2-endorigid algebras of finite representation type are exactly the standard 2-CY-tilted algebras
of finite representation type listed in~\cite{bo} (see also~\cite{l}) and the non-Jacobian
2-endorigid algebra of Section~\ref{subsection: not 2CY-tilted}.
\end{theorem}

\begin{proof}
The theorem follows from the classification \cite{a,bikr} of all standard 2-Calabi--Yau triangulated categories
with maximal rigid objects (see Table 1 and Table 2) and from the equivalences of categories in Proposition~\ref{prop: equivalences}.
\end{proof}

\begin{remark}
We note that the conclusion of Corollary~\ref{corollary: generalised standard} being weaker than one would like,
we do not know if the list discussed above contains all 2-endorigid algebras of finite representation type.
\end{remark}

\end{document}